\documentclass[onefignum,onetabnum]{siamart171218}

\usepackage{lipsum}
\usepackage{amsfonts}
\usepackage{graphicx}
\usepackage{epstopdf}
\usepackage{algorithm}
\usepackage{algorithmic}
\usepackage{tikz}
\usetikzlibrary{arrows}

\usepackage{multirow}
\usepackage{listings}
\usepackage{mathtools}
\usepackage{latexsym,amsmath,amsfonts,amscd}
\usepackage{subfigure}
\usepackage{bbm}
\newtheorem{rmk}{Remark}
\newtheorem{defi}{Definition}

\ifpdf
  \DeclareGraphicsExtensions{.eps,.pdf,.png,.jpg}
\else
  \DeclareGraphicsExtensions{.eps}
\fi

\headers{Discrete Maximum Principle of $C^0$-$Q^2$ FEM}{H. Li and X. Zhang}

\title{On the monotonicity and discrete maximum principle of the finite difference implementation of $C^0$-$Q^2$ finite element method
}

\author{Hao Li\thanks{Department of Mathematics,
Purdue University,
150 N. University Street,
West Lafayette, IN 47907-2067
  (\email{li2497@purdue.edu}, \email{zhan1966@purdue.edu}).}
\and Xiangxiong Zhang \footnotemark[2]}

\usepackage{amsopn}

\begin{document}

\maketitle

\begin{abstract}
We show that the fourth order accurate finite difference implementation of continuous finite element method with tensor product of quadratic polynomial basis is monotone thus satisfies the discrete maximum principle for solving a scalar variable coefficient equation $-\nabla\cdot(a\nabla u)+cu=f$ under a suitable mesh constraint. 
\end{abstract}



\section{Introduction}
\subsection{Monotonicity and discrete maximum principle}

Consider a Poisson equation with variable coefficients and Dirichlet boundary conditions on a two dimensional rectangular domain $\Omega=(0,1)\times(0,1)$:
\begin{equation}
\label{pde-1}
 \begin{split}
 \mathcal L u\equiv -\nabla\cdot(a \nabla u)+cu=0 & \quad \mbox{on} \quad \Omega, \\
 u=g & \quad \mbox{on}\quad  \partial\Omega, 
\end{split}
\end{equation}
where $a(x,y), c(x,y)\in C^0(\bar \Omega)$ with $0< a_{\min}\leq a(x,y)\leq a_{\max}$ and $c(x,y)\geq 0$.
For a smooth function $u\in C^2(\Omega)\cap C(\bar \Omega)$, maximum principle holds  \cite{evans}:
$ \mathcal L u\leq 0 \,\, \mbox{in} \,\, \Omega \Longrightarrow \max_{\bar \Omega} u\leq \max\left\{0, \max_{\partial \Omega} u\right\},$
and in particular, 
\begin{equation}\mathcal L u= 0 \,\, \mbox{in} \,\, \Omega \Longrightarrow    |u(x,y)|\leq \max\limits_{\partial \Omega} |u|, \quad \forall (x,y)\in \Omega.  \label{elliptic-mp}
\end{equation}
 
For various purposes, it is desired to have numerical schemes to satisfy \eqref{elliptic-mp} in the discrete sense.  
 A linear approximation to $\mathcal L$ can be represented as a matrix $L_h$. The matrix $L_h$ is called {\it monotone} if its inverse has nonnegative entries, i.e., $L_h^{-1}\geq 0$. All matrix inequalities in this paper are entrywise inequalities.
One sufficient condition for the discrete maximum principle is  the {\it monotonicity} of the scheme, which was also used to prove convergence of numerical schemes, e.g.,  \cite{bramble1962formulation, ciarlet1973maximum, axelsson1990monotonicity, ferket1996finite}.

 In this paper, we will discuss the monotonicity and discrete maximum principle of the simplest finite difference implementation of the continuous finite element method with $Q^2$ basis (i.e., tensor product of quadratic polynomial) for \eqref{pde-1}, which is a fourth order accurate scheme \cite{li2019fourth}. 

\subsection{Second accurate schemes and M-matrices}
The second order centered difference $u''\approx \frac{u_{i-1}-2u_i+u_{i+1}}{\Delta x^2}$ for solving $-u''(x)=f(x), u(0)=u(1)=0$ results in a tridiagonal $(-1,2,-1)$ matrix, which is an M-matrix. Nonsingular M-matrices are inverse-positive matrices and it is the most convenient  tool for constructing inverse-positive matrices. There are many equivalent definitions or characterizations of M-matrices, see 
\cite{plemmons1977m}. 
One convenient characterization of nonsingular M-matrices are nonsingular matrices with nonpositive off-diagonal entries and positive diagonal entries, and all row sums are non-negative with at least one row sum is positive. 

The continuous finite element method with piecewise linear basis forms an M-matrix for the variable coefficient problem \eqref{pde-1} on triangular meshes under reasonable mesh constraints \cite{xu1999monotone}. The M-matrix structure in linear finite element method also holds for a nonlinear elliptic equation \cite{karatson2009discrete}. 
For solving $-\Delta u=f$ on regular triangular meshes, linear finite element method reduces to the 5-point discrete Laplacian. Linear finite element method or the 5-point discrete Laplacian is the most popular method in the literature for constructing schemes satisfying a discrete maximum principle and bound-preserving properties. 
 
  Almost all high order accurate schemes result in positive off-diagonal entries in $L_h$ for solving $-\Delta u=f$ thus $L_h$ is no longer an M-matrix. The only known exceptions are the fourth order accurate 9-point discrete Laplacian and the fourth order accurate compact finite difference scheme. 

\subsection{Existing high order accurate monotone methods for two-dimensional Laplacian}

There are at least three kinds of  high order accurate schemes which have been proven to satisfy   $L^{-1}_h\geq 0$ for the Laplacian operator $\mathcal L u=-\Delta u$:

\begin{enumerate}
 \item Both the fourth order accurate 9-point discrete Laplacian scheme \cite{bramble1962formulation, bramble1963fourth} and  the fourth order accurate compact finite difference scheme \cite{lele1992compact, doi:10.1137/18M1208551} for $-\Delta u=f$ can be written as $S\mathbf u=W\mathbf f$ with $S$ being an M-matrix and $W\geq 0$,  thus $L_h^{-1}=S^{-1}M\geq 0$.
 \item
In \cite{bramble1964finite,bramble1964new},  Bramble and Hubbard constructed a fourth order accurate finite difference discrete Laplacian operator for which $L_h$ is not an M-matrix but monotonicity $L^{-1}_h\geq 0$ is ensured through an M-matrix factorization $L_h=M_1 M_2$, i.e., $L_h$ is a product of two M-matrices. 

\item Finite element method with quadratic polynomial (P2 FEM) basis on a regular triangular mesh can be implemented as a finite difference scheme defined at vertices and edge centers of triangles \cite{whiteman1975lagrangian}. The error estimate of P2 FEM is third order in $L^2$-norm. The stiffness matrix is not an M-matrix but its monotonicity was proven in \cite{lorenz1977inversmonotonie}.
\end{enumerate}

For discrete maximum principle to hold in P2 FEM on a generic triangular mesh,   
it was proven in \cite{hohn1981some} that it is  necessary and sufficient to require a very strong mesh constraint, which essentially gives either regular triangulation or equilateral triangulation. Thus discrete maximum principle holds in P2 FEM on  a regular triangulation or an equilateral triangulation. 
For finite element method with cubic and higher order polynomials  on regular triangular meshes, it was shown that discrete maximum principle fails in  \cite{vejchodsky2010angle}.

 \subsection{Other known results regarding discrete maximum principle}
For one-dimensional Laplacian,  discrete maximum principle was proven for arbitrarily high order finite element method using discrete Green's function in \cite{vejchodsky2007discrete}.
 The discrete Green's function was also used to analyze P1 FEM in two dimensions \cite{druaguanescu2005failure}. 
Discontinuous coefficients were considered and a nonlinear scheme was constructed in \cite{li2001maximum}. Piecewise constant coefficient in one dimension was considered in \cite{vejchodsky2007discrete-coef}. 
A numerical study for high order FEM with very accurate Gauss quadrature in two dimensions showed that  DMP was violated on non-uniform unstructured meshes for variable coefficients  in \cite{payette2012performance}.
A more general operator $\nabla(\mathbf a \nabla u)$ with matrix coefficients $\mathbf a$ was considered for linear FEM in \cite{10.1007/978-3-642-00464-3_43}. 
See  \cite{kuzmin2009constrained} for an anisotropic computational example. 

\subsection{Existing inverse-positive approaches when $L_h$ is not an M-matrix}

In this paper, we will focus on the finite difference implementation of continuous finite element method with $Q^2$ basis (Q2 FEM), which will be reviewed in Section \ref{sec-q2fdscheme}. 
The matrix $L_h$ in such a scheme is not an M-matrix due to its off-diagonal positive entries. There are at least three methods to study whether $L^{-1}_h\geq 0$ holds when M-matrix structure is lost:
\begin{enumerate}
\item An M-matrix factorization of the form $L_h=M_1M_2$ was shown in \cite{bramble1964new} and \cite{bohl1979inverse}. In Appendix \ref{appendix-a}, we will demonstrate an M-matrix factorization  for the finite difference implementation of $Q^2$ FEM solving $-\Delta u=f$. 
\item 
Perturbation of M-matrices by positive offdiagonal entries  without losing monotonicity was discussed in \cite{bouchon2007monotonicity}. 
\item In \cite{lorenz1977inversmonotonie}, Lorenz proposed a sufficient condition for ensuring $L_h=M_1M_2$. Lorenz's condition  will be reviewed in Section \ref{sec-lorenz}.
\end{enumerate}
The main result of this paper is to prove that $L_h^{-1}\geq 0$ and a discrete maximum principle holds under some mesh constraint in the fourth order accurate finite difference implementation of $Q^2$ FEM solving \eqref{pde-1} by verifying the Lorenz's condition.

\subsection{Extensions to discrete maximum principle for parabolic equations}

Classical solutions to the parabolic equation $u_t=\nabla\cdot(a \nabla u)$
 satisfy a maximum principle  \cite{evans}. 
 With suitable boundary conditions and initial value $u(x,y,0)$ 
 such as periodic or homogeneous Dirichlet boundary conditions and $\min\limits_\Omega u(x,y,0)=0$, the solution to the initial value problem satisfies the following maximum principle:
 \begin{equation}
\min_{(x,y)} u(x,y,0) \leq u(x,y,t)\leq \max_{(x,y)} u(x,y,0).
  \label{parabolic-mp}
\end{equation}

 Now consider solving $u_t=\nabla\cdot(a \nabla u)$ with  backward Euler time discretization, then $U^{n+1}$ satisfies an elliptic equation of the form \eqref{pde-1}:
 \begin{equation}
-\nabla\cdot(a \nabla U^{n+1})+\frac{1}{\Delta t}U^{n+1}=\frac{1}{\Delta t}U^{n}.
\label{backwardeuler}
 \end{equation}
 If $S_h$ denotes  spatial discretization for $-\nabla\cdot(a\nabla u)$, then the numerical scheme can be written as 
 $U^{n+1}=(I+\Delta t S_h)^{-1}U^n$. Let $\mathbf 1=\begin{bmatrix}1 & 1 &\cdots & 1                                                                                                                                                                                                                                                                                                                                                                                                                         \end{bmatrix}^T
$, then for suitable boundary conditions usually we have $S_h \mathbf 1=\mathbf 0$ since $S_h$ approximates a differential operator. So we have $(I+\Delta t S_h)\mathbf 1=\mathbf 1$ thus $(I+\Delta t S_h)^{-1}\mathbf 1=\mathbf 1$. 
If we further have the monotonicity $(I+\Delta t S_h)^{-1}\geq 0$, then each row of the $(I+\Delta t S_h)^{-1}$ has nonnegative entries and sums to one, thus the discrete maximum principle holds $\min_j U_j^n \leq U^{n+1}_j\leq \max_j U_j^n$, which is a desired and useful property in many applications.
For instance, second order centered difference or P1 finite element method has been used to construct schemes satisfying the discrete maximum principle in solving phase field equations 
\cite{tang2016implicit,shen2016maximum,xu2019stability}. 
 In the rest of the paper, we will only focus on discussing the equation \eqref{pde-1}, even though all discussions can be extended to solving the parabolic equation with backward Euler time discretization.  
\subsection{Contributions and organization of the paper}
To the best of our knowledge, this is the first time that a high order accurate scheme under suitable mesh constraints is proven to be monotone in the sense $L_h^{-1}\geq 0$ for solving a variable coefficient $a(\mathbf x)$ in \eqref{pde-1} in two dimensions. 
For simplicity, we only discuss an uniform mesh in this paper, even though the main results can be extended to non-uniform meshes. 
However, an additional mesh constraint is expected for discrete maximum principle to hold. See 
such a mesh constraint of non-uniform meshes for  Q1 FEM in \cite{christie1984maximum} and P2 FEM for one-dimensional problem in  \cite{vejchodsky2007discrete}.

This paper is organized as follows. In Section \ref{sec-q2fdscheme}, we describe the fourth order accurate finite difference implementation of $C^0$-$Q^2$ finite element method. In Section \ref{sec-cond}, we review the sufficient conditions to ensure monotonicity and discrete maximum principle. 
In Section \ref{sec-main}, we prove that the fourth order accurate finite difference implementation of $C^0$-$Q^2$ finite element method is monotone under some mesh constraints. Numerical tests are given in Section \ref{sec-test}.
Concluding remarks are given in Section \ref{sec-remark}.

\section{Finite difference implementation of $C^0$-$Q^2$ finite element method}
\label{sec-q2fdscheme}
Consider solving the following elliptic equation on $\Omega=(0,1)\times(0,1)$ with 
Dirichlet boundary conditions:
\begin{equation}
\label{pde-3}
 \begin{split}
 \mathcal L u\equiv -\nabla\cdot(a \nabla u)+cu=f & \quad \mbox{on} \quad \Omega, \\
 u=g & \quad \mbox{on}\quad  \partial\Omega. 
\end{split}
\end{equation}
Assume there is a function $\bar g \in H^1(\Omega)$ as an extension of $g$ so that $\bar g|_{\partial \Omega} = g$.
 The  variational form  of \eqref{pde-1} is to find $\tilde u = u - \bar g \in H_0^1(\Omega)$ satisfying
\begin{equation}\label{nonhom-var}
 \mathcal A(\tilde u, v)=(f,v) - \mathcal A(\bar g,v) ,\quad \forall v\in H_0^1(\Omega),
 \end{equation}
 where $\mathcal A(u,v)=\iint_{\Omega} a \nabla u \cdot \nabla v dx dy+\iint_{\Omega} c u v dx dy$, $ (f,v)=\iint_{\Omega}fv dxdy.$

   \begin{figure}[h]
 \subfigure[The  quadrature points and a FEM mesh]{\includegraphics[scale=0.8]{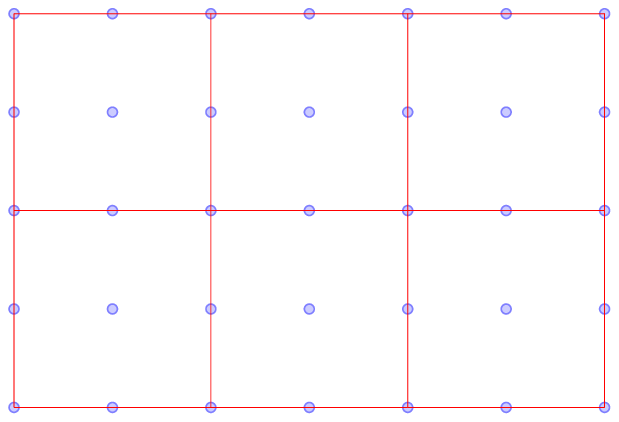} }
 \hspace{.6in}
 \subfigure[The corresponding finite difference grid]{\includegraphics[scale=0.8]{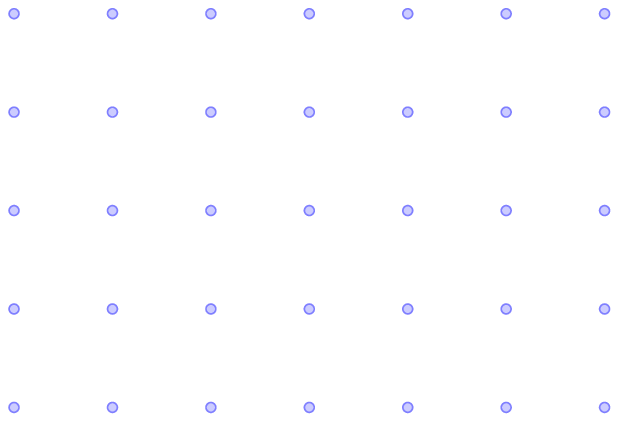}}
\caption{An illustration of $Q^2$ element and the $3\times3$ Gauss-Lobatto quadrature. }
\label{mesh}
 \end{figure}

 Let $h$ be the mesh size of the rectangular mesh and  $V_0^h\subseteq H^1_0(\Omega)$ be the continuous finite element space consisting of piecewise $Q^2$ polynomials (i.e., tensor product of piecewise quadratic polynomials), then the most convenient implementation of $C^0$-$Q^2$ finite element method is to use $3\times3$ Gauss-Lobatto quadrature rule for all the integrals,  see Figure \ref{mesh}.
 Such a numerical scheme can be defined as:  find $ u_h \in V_0^h$ satisfying
\begin{equation}\label{nonhom-var-num3}
\mathcal A_h( u_h, v_h)=\langle f,v_h \rangle_h - \mathcal A_h( g_I,v_h) ,\quad \forall v_h\in V_0^h,
\end{equation}
where   $\mathcal A_h(u_h,v_h)$  and $\langle f,v_h\rangle_h$ denote using tensor product of $3$-point Gauss Lobatto quadrature for integrals $\mathcal A(u_h,v_h)$ and $(f,v_h)$ respectively, and $g_I$ is the piecewise $Q^2$ Lagrangian interpolation polynomial at the $3\times 3$ quadrature points shown in 
Figure \ref{mesh} of the following function:
\[g(x,y)=\begin{cases}
   0,& \mbox{if}\quad (x,y)\in (0,1)\times(0,1),\\
   g(x,y),& \mbox{if}\quad (x,y)\in \partial\Omega.\\
  \end{cases}
\] 
Then $\bar u_h =   u_h + g_I$ is the numerical solution for the problem \eqref{pde-3}.
We emphasize that \eqref{nonhom-var-num3} is not a straightforward approximation to \eqref{nonhom-var} since $\bar g$ is never used. 
It was proven in \cite{li2019fourth} that the  scheme  \eqref{nonhom-var-num3} is fourth order accurate 
if coefficients and exact solutions are smooth. 
Notice that $\bar u_h$ satisfies:
\begin{equation}
\mathcal A_h(  \bar u_h, v_h)=\langle f,v_h \rangle_h,\quad \forall v_h\in V_0^h.
\label{scheme-1}
\end{equation}
See \cite{li2019fourth} for the detailed finite difference implementation and proof of fourth order accuracy for the scheme \eqref{nonhom-var-num3}. 
\subsection{One-dimensional case}

Now consider the one-dimensional Dirichlet boundary value problem:
\begin{align*}
-(a u')' +c u= & f \textrm{ on } (0,1), \\
 u(0) = \sigma_0,  \quad & u(1) =  \sigma_1.
\end{align*}

Consider a uniform mesh $x_i = ih$, $i = 0,1,\dots, n+1 $, $h=\frac{1}{n+1}$. Assume $n$ is odd and let $M=\frac{n+1}{2}$. Define intervals $I_k =[x_{2k},x_{2k+2}]$ for $k=0,\dots,M-1$ as a finite element mesh for $P^2$ basis. Define 
$$V^h=\{v\in C^0([0,1]): v\in P^2(I_k), k = 0,\dots, M-1\}.$$ Let $\{\phi_i\}_{i=0}^{n+1} 	\subset V^h $ be a basis for $V^h$ so that $\phi_i(x_j)= \delta_{ij}, \,i,j=0,1,\dots,n+1$. Let $u_i=u_h(x_i)$, $u_0=\sigma_0$ and $u_{n+1}=\sigma_1$, then $u_h, \bar u_h\in V^h$ can be represented as
\[u_h(x)=\sum_{i=1}^{n}u_i \phi_i(x), \quad \bar u_h(x)=\sum_{i=0}^{n+1}u_i \phi_i(x). \]

Let $f_j=f(x_j)$, then \eqref{scheme-1} becomes 
\[\langle a u_h',\phi_i'\rangle_h+\langle c u_h,\phi_i\rangle_h = \langle f,\phi_i\rangle_h, \quad i=1,\dots,n; u_0=\sigma_0, u_{n+1}=\sigma_1, \]
which are
$$\sum_{j=0}^{n+1} u_j \left(\langle a \phi_j',\phi_i'\rangle_h+ \langle c \phi_j,\phi_i\rangle_h\right) =\sum_{j=0}^{n+1} f_j \langle  \phi_j,\phi_i\rangle_h, \quad i=1,\dots,n; u_0=\sigma_0, u_{n+1}=\sigma_1.$$
The matrix form is $S\bar{\mathbf u}=M\bar{\mathbf f}$ where 
\[\bar{\mathbf u}=\begin{bmatrix}
                 u_0 & u_1 & u_2 \cdots & u_n & u_{n+1}  
                  \end{bmatrix}^T,\quad 
                  \bar {\mathbf f}=\begin{bmatrix}
                \sigma_0 & f_1 & f_2 \cdots & f_n  & \sigma_1
                  \end{bmatrix}^T.
\]
 The scheme can be written as $\mathcal L_h (\bar{\mathbf u})=\bar{\mathbf f}$. The linear operator $\mathcal L_h$ has the matrix representation 
 $L_h=M^{-1}S$.
 
 For the Laplacian $\mathcal L u=-u''$, we have
 \begin{subequations}
 \label{p2fd-laplacian-1dscheme}
 \begin{align}
  \mathcal L_h (\bar{\mathbf u})_0&=u_0=\sigma_0, \quad 
  \mathcal L_h (\bar{\mathbf u})_{n+1}=u_{n+1}=\sigma_1, \\
  \mathcal L_h (\bar{\mathbf u})_i&=\frac{-u_{i-1}+2u_i-u_{i+1}}{h^2}=f_i, \quad \mbox{if $i$ is odd, i.e., $x_i$ is a cell center},\\
  \mathcal L_h (\bar{\mathbf u})_i&=
\frac{u_{i-2}-8u_{i-1}+14u_i-8u_{i+1}+u_{i+2}}{4h^2}=f_i,\quad \mbox{if $i$ is even, i.e., $x_i$ is a cell end}.
  \end{align}
 \end{subequations}

 For the variable coefficient operator $\mathcal L u=-(a u')'+cu$, we have 
 
  \begin{subequations}
 \label{p2fd-vcoef-1dscheme}
 \begin{equation}
  \label{p2fd-vcoef-1dscheme-1}
  \mathcal L_h (\bar{\mathbf u})_0 =u_0=\sigma_0, \quad 
  \mathcal L_h (\bar{\mathbf u})_{n+1}=u_{n+1}=\sigma_1,
 \end{equation}
  and if $x_i$ is a cell center, we have
   \begin{equation}
     \label{p2fd-vcoef-1dscheme-2}
  \mathcal L_h (\bar{\mathbf u})_i=\frac{-(3a_{i-1}+a_{i+1})u_{i-1}+4(a_{i-1}+a_{i+1})u_i-(a_{i-1}+3a_{i+1})u_{i+1}}{4h^2}+c_i u_i=f_i;
 \end{equation}
  and if $x_i$ is a cell end, then 
   \begin{align}  
  \notag
\mathcal L_h (\bar{\mathbf u})_i=\frac{(3a_{i-2}-4a_{i-1}+3a_i)u_{i-2}-(4a_{i-2}+12a_i)u_{i-1}+(a_{i-2}+4a_{i-1}+18 a_i+4a_{i+1}+a_{i+2})u_i}{8h^2}\\
   +\frac{-(12a_{i}+4a_{i+2})u_{i+1}+(3a_{i+2}-4a_{i+1}+3 a_i)u_{i+2}}{8h^2}+c_iu_i=f_i.  \label{p2fd-vcoef-1dscheme-3} \end{align}
    \end{subequations}

 \subsection{Two-dimensional case}  
 
Consider a uniform grid $(x_i,y_j)$ for a rectangular domain $[0,1]\times[0,1]$
where
$x_i = ih$, $i = 0,1,\dots, n+1$
and
$y_j = jh$, $j = 0,1,\dots, n+1$, $h=\frac{1}{n+1}$, where $n$ must be odd.
Let $u_{ij}$ denote the numerical solution at $(x_i, y_j)$.
Let $\mathbf u$ denote an abstract vector consisting of $u_{ij}$ for $i,j=1,2,\cdots,n$. Let $\bar{\mathbf u}$ denote an abstract vector consisting of $u_{ij}$ for $i,j=0,1,2,\cdots,n,n+1$.
Let $\bar{\mathbf f}$ denote an abstract vector consisting of $f_{ij}$ for $i,j=1,2,\cdots,n$ and the boundary condition $g$ at the boundary grid points. 

The scheme \eqref{scheme-1} for solving \eqref{pde-3} can still be written as 
$\mathcal L_h(\bar{\mathbf u})=\bar{\mathbf f}$.

 \subsubsection{Two-dimensional  Laplacian}  
For the Laplacian $\mathcal L u=-\Delta u$, $\mathcal L_h(\bar{\mathbf u})$ can be expressed as the following.
If $(x_i, y_j)\in \partial \Omega$, then 
 \[ \mathcal L_h(\bar{\mathbf u})_{i,j}=u_{i,j}=g_{i,j}.\]
If $(x_i, y_j)$ is an interior grid point 
and a cell center , $\mathcal L_h(\bar{\mathbf u})_{i,j}$ is equal to
\begin{subequations}
\label{q2fd-2D-laplacian}
\begin{equation}
\frac{-u_{i-1,j}-u_{i+1,j}+4u_{i,j}-u_{i,j+1}-u_{i+1,j}}{h^2}=f_{i,j}.
\end{equation}
For interior grid points, there are three types: cell center, edge center and knots. See Figure \ref{fig-points}. 
If $(x_i, y_j)$ is an interior grid point 
and an edge center for an edge parallel to x-axis, $\mathcal L_h(\bar{\mathbf u})_{i,j}$ is equal to
\begin{equation}
\frac{-u_{i-1,j}+2u_{i,j}-u_{i+1,j}}{h^2}+\frac{u_{i,j-2}-8u_{i,j-1}+14u_{i,j}-8u_{i,j+1}+u_{i,j+2}}{4h^2}=f_{i,j}.
\end{equation}
If $(x_i, y_j)$ is an interior grid point 
and an edge center for an edge parallel to y-axis, $\mathcal L_h(\bar{\mathbf u})_{i,j}$ is similarly defined as above.
If $(x_i, y_j)$ is an interior grid point 
and a knot  $(x_i, y_j)$, $\mathcal L_h(\bar{\mathbf u})_{i,j}$ is equal to
\begin{equation}
\frac{u_{i-2,j}-8u_{i-1,j}+14u_{i,j}-8u_{i+1,j}+u_{i+2,j}}{4h^2}+\frac{u_{i,j-2}-8u_{i,j-1}+14u_{i,j}-8u_{i,j+1}+u_{i,j+2}}{4h^2}=f_{i,j}.
\end{equation}
\end{subequations}

\begin{figure}
\begin{center}
 \scalebox{0.7}{
\begin{tikzpicture}[samples=100, domain=-3:3, place/.style={circle,draw=blue!50,fill=blue,thick,
inner sep=0pt,minimum size=1.5mm},transition/.style={circle,draw=red,fill=red,thick,inner sep=0pt,minimum size=2mm}
,point/.style={circle,draw=black,fill=black,thick,inner sep=0pt,minimum size=2mm}]

\draw[color=red] (-2,-2)--(-2,2);
\draw[color=red] (-2,-2)--(2,-2);
\draw[color=red] (2,-2)--(2, 2);
\draw[color=red] (-2,2)--(2, 2);
\node at ( -2,-2) [place] {};
\node at ( 0,-2) [point] {};
\node at ( 2,-2) [place] {};

\node at ( -2,0) [point] {};
\node at ( 0,0) [transition] {};
\node at ( 2,0) [point] {};

\node at ( -2,2) [place] {};
\node at ( 0,2) [point] {};
\node at ( 2,2) [place] {};
 \end{tikzpicture}
}
\caption{Three types of interior grid points: red cell center, blue knots and black edge centers for a finite element cell. }
\end{center}
\label{fig-points}
\end{figure}
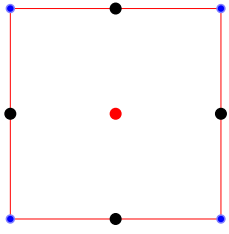

If ignoring the denominator $h^2$, then the stencil of the operator $\mathcal L_h$ 
at interior grid points can be represented as:
 \[ \quad \mbox{  cell center}  \begin{array}{ccc}
& -1& \\
-1 & 4 & -1\\
& -1& 
\end{array}\qquad
\mbox{knots}  \begin{array}{ccccc}
&& \frac14& &\\
&& -2& &\\
\frac14& -2 & 7 & -2 &\frac14\\
&& -2& &\\
&& \frac14& &
\end{array}
\]
\[
\mbox{edge center (edge parallel to $y$-axis)}  \begin{array}{ccccc}
&& -1& &\\
\frac14& -2 & \frac{11}{2} & -2 &\frac14\\
&& -1& &
\end{array}\]
\[
\mbox{edge center (edge parallel to $x$-axis)}  \begin{array}{ccccc}
&& \frac14& &\\
&& -2& &\\
& -1 & \frac{11}{2} & -1 &\\
&& -2& &\\
&& \frac14& &
\end{array}\]
\subsection{Two-dimensional variable coefficient case}
For $\mathcal L u=- \nabla\cdot (a \nabla u)+cu$, $\mathcal L_h(\bar{\mathbf u})$ will have exactly the same stencil size  as the Laplacian case. At boundary points $(x_i, y_j)\in \partial \Omega$, $\mathcal L_h(\bar{\mathbf u})=\bar{\mathbf f}$ becomes 
\begin{subequations}
   \label{p2fd-vcoef-2dscheme}
 \begin{equation}
  \label{p2fd-vcoef-2dscheme-1}
 \mathcal L_h(\bar{\mathbf u})_{i,j}=u_{i,j}=g_{i,j}.
 \end{equation}
  
  If $(x_i, y_j)$ is an interior grid point 
and a cell center, $\mathcal L_h(\bar{\mathbf u})_{i,j}$ is equal to
\begin{align}
\frac{-(3a_{i-1,j}+a_{i+1,j})u_{i-1,j}+4(a_{i-1,j}+a_{i+1,j})u_{i,j}-(a_{i-1,j}+3a_{i+1,j})u_{i+1,j}}{4h^2}\\
+\frac{-(3a_{i,j-1}+a_{i,j+1})u_{i,j-1} 
+4(a_{i,j-1}+a_{i,j+1})u_{i,j}-(a_{i,j-1}+3a_{i,j+1})u_{i,j+1}}{4h^2}+c_{ij}u_{ij}.  \notag 
 \end{align}

If $(x_i, y_j)$ is an interior grid point 
and a knot, $\mathcal L_h(\bar{\mathbf u})_{i,j}$ is equal to
\begin{align}
\frac{(3a_{i-2,j}-4a_{i-1,j}+3a_{i,j})u_{i-2,j}-(4a_{i-2,j}+12a_{i,j})u_{i-1,j}+(a_{i-2,j}+4a_{i-1,j}+18 a_{i,j}+4a_{i+1,j}+a_{i+2,j})u_{i,j}}{8h^2} \\
+\frac{-(12a_{i,j}+4a_{i+2,j})u_{i+1,j}+(3a_{i+2,j}-4a_{i+1,j}+3 a_{i,j})u_{i+2,j}}{8h^2} \notag  \\
+\frac{(3a_{i,j-2}-4a_{i,j-1}+3a_{i,j})u_{i,j-2}-(4a_{i,j-2}+12a_{i,j})u_{i,j-1}+(a_{i,j-2}+4a_{i,j-1}+18 a_{i,j}+4a_{i,j+1}+a_{i,j+2})u_{i,j}}{8h^2} \notag \\
+\frac{-(12a_{i,j}+4a_{i,j+2})u_{i,j+1}+(3a_{i,j+2}-4a_{i,j+1}+3 a_{i,j})u_{i,j+2}}{8h^2}+c_{ij}u_{ij}. \notag 
 \end{align}

If $(x_i, y_j)$ is an interior grid point 
and an edge center for an edge parallel to $y$-axis, $\mathcal L_h(\bar{\mathbf u})_{i,j}$ is equal to
\begin{align}\frac{(3a_{i-2,j}-4a_{i-1,j}+3a_{i,j})u_{i-2,j}-(4a_{i-2,j}+12a_{i,j})u_{i-1,j}+(a_{i-2,j}+4a_{i-1,j}+18 a_{i,j}+4a_{i+1,j}+a_{i+2,j})u_{i,j}}{8h^2}\notag \\
+\frac{-(12a_{i,j}+4a_{i+2,j})u_{i+1,j}+(3a_{i+2,j}-4a_{i+1,j}+3 a_{i,j})u_{i+2,j}}{8h^2}\\
+\frac{-(3a_{i,j-1}+a_{i,j+1})u_{i,j-1} 
+4(a_{i,j-1}+a_{i,j+1})u_{i,j}-(a_{i,j-1}+3a_{i,j+1})u_{i,j+1}}{4h^2}
 +c_{ij}u_{ij}. \notag 
 \end{align}
  If $(x_i, y_j)$ is an interior grid point 
and an edge center for an edge parallel to $x$-axis, 
 $\mathcal L_h(\bar{\mathbf u})_{i,j}$ is equal to
\begin{align}\frac{(3a_{i,j-2}-4a_{i,j-1}+3a_{i,j})u_{i,j-2}-(4a_{i,j-2}+12a_{i,j})u_{i,j-1}+(a_{i,j-2}+4a_{i,j-1}+18 a_{i,j}+4a_{i,j+1}+a_{i,j+2})u_{i,j}}{8h^2}\notag \\
+\frac{-(12a_{i,j}+4a_{i,j+2})u_{i,j+1}+(3a_{i,j+2}-4a_{i,j+1}+3 a_{i,j})u_{i,j+2}}{8h^2}\\
+\frac{-(3a_{i-1,j}+a_{i+1,j})u_{i-1,j} 
+4(a_{i-1,j}+a_{i+1,j})u_{i,j}-(a_{i-1,j}+3a_{i+1,j})u_{i+1,j}}{4h^2}
 +c_{ij}u_{ij}. \notag 
 \end{align}
 
 \end{subequations}

\section{Sufficient conditions for monotonicity and discrete maximum principle} 
\label{sec-cond}
\subsection{Discrete maximum principle} 

Assume there are $N$ grid points in the domain $\Omega$ and $N^\partial$ grid points on 
$\partial\Omega$.
Define
$$\mathbf u=\begin{pmatrix}
u_1& u_2& \cdots& u_N             
            \end{pmatrix}^T, \quad \mathbf u^\partial=\begin{pmatrix}
u^\partial_1& u^\partial_2& \cdots& u^\partial_{N^\partial}             
            \end{pmatrix}^T,$$
\[\tilde{\mathbf u}=\begin{pmatrix}
u_1& u_2& \cdots& u_N & u^\partial_1& u^\partial_2& \cdots& u^\partial_{N^\partial}            
            \end{pmatrix}^T.\]
 A finite difference scheme can be written as 
 \begin{align*}
\mathcal L_h (\tilde{\mathbf u})_i=\sum_{j=1}^N b_{ij}u_j+\sum_{j=1}^{N^\partial}  b^\partial_{ij}u^\partial_j=&f_i,\quad 1\leq i\leq N,\\
  u^\partial_i=&g_i, \quad 1\leq i\leq N^\partial.  
 \end{align*}
 The matrix form is 
 \[\tilde L_h \tilde{\mathbf u}=\tilde{\mathbf f},\tilde L_h =\begin{pmatrix}
                       L_h & B^\partial\\
                       0 & I
                      \end{pmatrix}, \tilde{\mathbf u}=\begin{pmatrix}
                       \mathbf u\\ \mathbf u^\partial \end{pmatrix},
                       \tilde{\mathbf f}=\begin{pmatrix}
                       \mathbf f\\ \mathbf g \end{pmatrix}.
\]
The discrete maximum principle is 
\begin{equation}
\mathcal L_h (\tilde{\mathbf u})_i\leq 0, 1\leq i \leq N\Longrightarrow \max_i u_i\leq \max\{0, \max_i u_i^\partial\} 
\label{dmp}
\end{equation}
which implies 
\[\mathcal L_h (\tilde{\mathbf u})_i= 0, 1\leq i \leq N\Longrightarrow |u_i|\leq \max_i |u_i^\partial|.   \]

 The following result was proven in \cite{ciarlet1970discrete}:
\begin{theorem}
 A finite difference operator $\mathcal L_h$ satisfies the discrete maximum principle 
 \eqref{dmp} if  $\tilde L_h^{-1}\geq 0$ and  all row sums of $\tilde L_h$ are non-negative. 
 \end{theorem}

Let $\bar{\mathbf u}$ and $\bar{\mathbf f}$ be the same vectors as defined in 
Section \ref{sec-q2fdscheme}. For the same finite difference scheme, the matrix form can also be written as
\[\bar{L}_h \bar{\mathbf u}=\bar{\mathbf f}.\]
Notice that there exist two  permutation matrices $P_1$ and $P_2$ such that $\bar{\mathbf u}=P_1\tilde{\mathbf u}$ and $\bar{\mathbf f}=P_2\tilde{\mathbf f}$. Since the matrix vector form of the same scheme is also $\tilde L_h \tilde{\mathbf u}=\tilde{\mathbf f}$, we obtain $P_2^{-1}\bar{L}_hP_1=\tilde L_h$. 
Notice that a permutation matrix $P$ is  inverse-positive and the signs of row sums will not be altered after multiplying $P$ to $\tilde L_h$. Thus we have 
\begin{theorem}
 If $\bar{L}_h$ is inverse-positive and row sums of  $\bar{L}_h$ are non-negative, then  $\mathcal L_h$ satisfies the discrete maximum principle
 \eqref{dmp}. 
\end{theorem}
Notice that $\tilde L_h^{-1} =\begin{pmatrix}
                       L_h^{-1} & -L_h^{-1}B^\partial\\
                       0 & I
                      \end{pmatrix}$, thus we have
\begin{theorem}
\label{theorem-fullandsmallmatrix}
 If $\bar{L}_h^{-1}\geq 0$, then $\tilde L_h^{-1}\geq 0 $ thus ${L}_h^{-1}\geq 0$. 
\end{theorem}

Let $\mathbf 1$ denote a vector of suitable size with $1$ as entries, then for all schemes in Section \ref{sec-q2fdscheme},  $\mathcal L_h(\mathbf 1)\geq 0$, which implies the row sums of $\bar{L}_h$ are non-negative. Thus from now on, we only need to discuss the monotonicity of the matrix $\bar{L}_h$.
\subsection{Characterizations of nonsingular M-matrices}
M-matrices belong to the set of Z-matrices which are matrices  with nonpositive off-diagonal entries. 
Nonsingular M-matrices are always inverse-positive. See \cite{plemmons1977m} for the definition and various characterization of nonsingular M-matrices.
The following is a convenient sufficient condition to characterize nonsingular M-matrices:
\begin{theorem}
\label{rowsumcondition-thm}
For a real square matrix $A$  with positive diagonal entries and non-positive off-diagonal entries, $A$ is a nonsingular M-matrix if and only if all the row sums of $A$ are non-negative and at least one row sum is positive. 
\end{theorem}
\begin{proof}
By condition $C_{10}$ in \cite{plemmons1977m}, $A$ is a nonsingular M-matrix if and only if $A+a I$ is nonsingular for any $a\geq 0$.  
Since all the row sums of $A$ are non-negative and at least one row sum is positive,
the matrix $A$ is  
irreducibly diagonally dominant thus nonsingular, and $A+a I$ is strictly diagonally dominant thus nonsingular for any $a>0.$
\end{proof}

\begin{defi}
Let $\mathcal N = \{1,2,\dots,n\}$. For $\mathcal N_1, \mathcal N_2 \subset \mathcal N$, we say a matrix $A$ of size $n\times n$ connects $\mathcal N_1$ with $\mathcal N_2$ if 
\begin{equation}
\forall i_0 \in \mathcal N_1, \exists i_r\in \mathcal N_2, \exists i_1,\dots,i_{r-1}\in \mathcal N \quad \mbox{s.t.}\quad  a_{i_{k-1}i_k}\neq 0,\quad k=1,\cdots,r.
\label{condition-connect}
\end{equation}
If perceiving $A$ as a directed graph adjacency matrix of vertices labeled by $\mathcal N$, then \eqref{condition-connect} simply means that there exists a directed path from any vertex in $\mathcal N_1$ to at least one vertex in $\mathcal N_2$.  
In particular, if $\mathcal N_1=\emptyset$, then any matrix $A$  connects $\mathcal N_1$ with $\mathcal N_2$.
\end{defi}

Given a square matrix $A$ and a column vector $\mathbf x$, we define
\[\mathcal N^0(A\mathbf x)=\{i: (A\mathbf x)_i=0\},\quad 
\mathcal N^+(A\mathbf x)=\{i: (A\mathbf x)_i>0\}.\]

By condition $L_{36}$ in \cite{plemmons1977m}, we have the following characterization of nonsingular M-matrices:
\begin{theorem}
\label{thm-l36}
For a real square matrix $A$  with non-positive off-diagonal entries, if there is a vector $\mathbf x>0$  with $A\mathbf x\geq 0$ s.t. $A$ connects $\mathcal N^0(A\mathbf x)$ with
$\mathcal N^+(A\mathbf x)$, then $A$ is a nonsingular M-matrix thus $A^{-1}\geq 0$. 
\end{theorem}

\subsection{Lorenz's sufficient condition for monotonicity}
\label{sec-lorenz}
All results in this subsection were first shown in  \cite{lorenz1977inversmonotonie}. For completeness, we include detailed proof.

Given a matrix $A=[a_{ij}]\in \mathbbm{R}^{n\times n}$, define its diagonal, positive and negative off-diagonal parts as $n\times n$ matrices $A_d$, $A_a$, $A_a^+$, $A_a^-$:
\[(A_d)_{ij}=\begin{cases}
a_{ii}, & \mbox{if} \quad i=j\\
0, & \mbox{if} \quad  i\neq j
\end{cases}, \quad A_a=A-A_d,
\]
\[(A_a^+)_{ij}=\begin{cases}
a_{ij}, & \mbox{if} \quad a_{ij}>0,\quad i\neq j\\
0, & \mbox{otherwise}.
\end{cases}, \quad A_a^-=A_a-A^+_a.
\]

\begin{lemma}\label{monotonelemma}
If $A$ is monotone, then for any two matrices $B\geq C$, $A^{-1}B\geq A^{-1}C$.
\end{lemma}
\begin{proof}
For any two column vectors $\mathbf b \geq \mathbf c$, we have
$$\mathbf{b} - \mathbf c\geq 0\Rightarrow A^{-1} (\mathbf{b}- \mathbf c)\geq 0\Rightarrow A^{-1}\mathbf b\geq A\mathbf c.$$
By considering $\mathbf b$ and $\mathbf c$ as   column vectors of $B$ and $C$, we get $A^{-1}B\geq A^{-1}C$.
\end{proof}

\begin{lemma}\label{diaglemma}
If $A$ is an M-matrix, then $A_d \geq A$ and $A^{-1}\geq A_d^{-1}$.  
\end{lemma}
\begin{proof}
$A_d \geq A$ is trivial. $A$ is monotone, thus 
$$A_d \geq A \Rightarrow A^{-1}A_d\geq A^{-1}A=I.$$
And $A_d^{-1}\geq 0$ implies 
$$A^{-1}A_d\geq I \Rightarrow A^{-1}A_dA_d^{-1}\geq IA_d^{-1}\Rightarrow A^{-1}\geq A_d^{-1}.$$
\end{proof}

\begin{theorem}\label{thm1}
If $A_a \leq 0$ and there exists a nonzero vector $\mathbf e\in \mathbbm{R}^n$ such that $\mathbf e\geq 0$ and $A\mathbf e \geq 0$. Moreover, $A$ connects $\mathcal N^0(A\mathbf e)$ with $\mathcal N^+(A\mathbf e)$. Then the following hold:
\begin{itemize}
\item $\mathbf e > 0$.
\item $a_{ii} > 0$, $\forall i \in N$.
\item $A$ is a M-matrix and $A^{-1}\geq 0$.
\end{itemize}
\end{theorem}
\begin{proof}
Assume there is  one index $i$ such that $e_i=0$, then 
\[0\leq (A\mathbf e)_i=\sum_{j\neq i} a_{ij}e_j \leq0\Rightarrow (A\mathbf e)_i=0
\Rightarrow \sum_{j\neq i} a_{ij}e_j=0\Rightarrow a_{ij}e_j=0, \forall j.\]
Thus if $a_{ij}<0$, then $e_j=0$, which  implies $(A\mathbf e)_j=0$ by the same argument as above. Therefore, $A$ has no off-diagonal nonzero entry $a_{kl}$ such that $k\in \mathcal N^0(A\mathbf e)$  and $l\in \mathcal N^+(A\mathbf e)$. In other words, if $A$ represents the graph adjacency matrix for a directed graph of vertices indexed by $1,2,\cdots, n$, then 
 any edge starting from a vertex $i\in  \mathcal N^0(A\mathbf e)$ points to vertices in $\mathcal N^0(A\mathbf e)$, thus there is no directed path  from $i\in  \mathcal N^0(A\mathbf e)$ to any vertex in 
 $\mathcal N^+(A\mathbf e)$,
which  contradicts to the assumption that $A$ connects $\mathcal N^0(A\mathbf e)$ with $\mathcal N^+(A\mathbf e)$. 
 With $\mathbf e>0$, the rest is proven by following Theorem \ref{thm-l36}.  
\end{proof}

\begin{corollary}\label{cor1}
If $A$ is a nonsingular M-matrix, $\mathbf f\in \mathbbm{R}^n$ is a nonzero vector with $\mathbf f\geq 0$ and $A$ connects $\mathcal N^0(\mathbf f)$ with $\mathcal N^+(\mathbf f)$, then $A^{-1}\mathbf f > 0$.
\end{corollary}
\begin{proof}
By using $\mathbf e=A^{-1}\mathbf f\geq 0$ in Theorem \ref{thm1}, we get $A^{-1}\mathbf f > 0$.
\end{proof}
\begin{theorem}\label{thm2}
If $A\leq M_1M_2\cdots M_k L$ where $M_1, \cdots, M_k$ are nonsingular M-matrices and $L_a\leq 0$,  and there exists a nonzero vector $\mathbf e\geq 0$ such that one of the matrices $M_1, \cdots, M_k , L$ connects $\mathcal N^0(A\mathbf e)$ with $\mathcal N^+(A\mathbf e)$. Then $A$ is a product of $k+1$ nonsingular M-matrices thus $A^{-1}\geq 0$. 
\end{theorem}
\begin{proof}
Let $M= M_1M_2\cdots M_k$, then $M$ is monotone. By Lemma \ref{monotonelemma}, we get
\begin{equation}\label{m-1aleql}
M^{-1}A\leq L,
\end{equation}
thus
\begin{equation}\label{m-1aoffdiag}
(M^{-1}A)_a \leq 0. 
\end{equation}
For each $M_i$, $i = 1,\dots,k,$ by Lemma \ref{diaglemma}, we have 
\begin{equation}\label{invdiag}
(M_i)^{-1} \geq ((M_i)_d)^{-1}\Rightarrow M^{-1}\geq  (M_k)_d^{-1}\cdots  (M_1)_d^{-1},
\end{equation}
which implies 
\begin{equation}\label{m-1aepos}
M^{-1}A\mathbf e \geq c A\mathbf e,
\end{equation}
 for some positive number $c$.

If $L$ connects $\mathcal N^0(A\mathbf e)$ with $\mathcal N^+(A\mathbf e)$, then  $M^{-1}A$ also connects $\mathcal N^0(A\mathbf e)$ with $\mathcal N^+(A\mathbf e)$ because \eqref{m-1aleql} implies that $(M^{-1}A)_{ij}\neq 0$ whenever $L_{ij}\neq0$ for any $i\neq j$. By \eqref{m-1aepos}, 
$\mathcal N^+(A\mathbf e)\subset \mathcal N^+(M^{-1}A\mathbf e)$ and  $\mathcal N^0(M^{-1}A\mathbf e)\subset \mathcal N^0(A\mathbf e)$, thus $M^{-1}A$ also connects $\mathcal N^0(M^{-1}A\mathbf e)$ with $\mathcal N^+(M^{-1}A\mathbf e)$. With \eqref{m-1aoffdiag}, by  Theorem \ref{thm1}, $M^{-1}A$ is a nonsingular M-matrix thus $A$ is a product of $k+1$ M-matrices which implies $A$ is monotone. 

If $M_i$ connects $\mathcal N^0(A\mathbf e)$ with $\mathcal N^+(A\mathbf e)$ for some $1\leq i \leq k$. Let $M' = M_{1}\dots M_{i-1}$. Similar to \eqref{invdiag} and \eqref{m-1aepos}, we get
\begin{equation}
(M')^{-1}A\mathbf e \geq c_2 A\mathbf e,\quad c_2> 0,
\end{equation}
which implies that $M_i$ connects $\mathcal N^0((M')^{-1}Ae)$ with $\mathcal N^+((M')^{-1}A\mathbf e)$. By Corollary \ref{cor1}, we know $M^{-1}_i(M')^{-1}A\mathbf e > 0$, thus $M^{-1}A\mathbf e > 0$. With \eqref{m-1aoffdiag}, by Theorem \ref{thm1}  $M^{-1}A$ is a M-matrix thus $A$ is a product of $k+1$ M-matrices which implies $A$ is monotone. 
\end{proof}
\begin{theorem}\label{thm3}
If $A^-_a$ has a decomposition: $A^-_a = A^z + A^s = (a_{ij}^z) + (a_{ij}^s)$ with $A^s\leq 0$ and $A^z \leq 0$, such that 
\begin{subequations}
 \label{lorenz-condition}
\begin{align}
A_d + A^z \textrm{ is a nonsingular M-matrix},\label{cond1}\\ 
A^+_a \leq A^zA^{-1}_dA^s \textrm{ or equivalently } \forall a_{ij} > 0 \textrm{ with } i \neq j, a_{ij} \leq \sum_{k=1}^n a_{ik}^za_{kk}^{-1}a_{kj}^s,\label{cond2}\\
\exists \mathbf e \in \mathbbm{R}^n\setminus\{\mathbf 0\}, \mathbf e\geq 0 \textrm{ with $A\mathbf e \geq 0$ s.t. $A^z$ or $A^s$  connects $\mathcal N^0(A\mathbf e)$ with $\mathcal N^+(A\mathbf e)$.} \label{cond3}
\end{align}
\end{subequations}
Then $A$ is a product of two nonsingular M-matrices thus $A^{-1}\geq 0$.
\end{theorem}
\begin{proof}
By \eqref{cond2}, we have 
\begin{equation}
A=A_d+A^z+A^s+A^+_a\leq (A_d + A^z)(I+A^{-1}_dA^s).
\end{equation}
By \eqref{cond3}, either $A_d + A^z$ or $I+A^{-1}_dA^s$ connects $\mathcal N^0(A\mathbf e)$ with $\mathcal N^+(A\mathbf e)$. 
By applying Theorem \ref{thm2} for the case $k=1$, $M_1=A_d+A^z$ and $L=I+A^{-1}_dA^s$, we get $A^{-1}\geq 0$. 
\end{proof}

\section{The main result}
\label{sec-main}

For a general matrix, conditions \eqref{lorenz-condition} in Theorem \ref{thm3} can be difficult to verify. We will first derive a simplified version of Theorem \ref{thm3} then verify it for the schemes in  Section \ref{sec-q2fdscheme}. 
\subsection{A simplified sufficient condition for monotonicity}

We will take advantage of the directed graph described by the 5-point discrete Laplacian, i.e., the second order centered difference scheme, which has similar off-diagonal negative entry patterns as the schemes in Section \ref{sec-q2fdscheme}. 

For the one-dimensional problem $-u''=f, x\in(0,1)$ with $u(0)=u(1)$, the scheme can be written as $
u_0=\sigma_0,  
 u_{n+1}=\sigma_1, \frac{-u_{i-1}+2u_i-u_{i+1}}{h^2}=f_i, i=1,\cdots, n.$
The matrix vector form is $K \bar{\mathbf u}=\bar{\mathbf f}$ where
\begin{equation}
K=\frac{1}{h^2}\begin{psmallmatrix}
   h^2 &&&&&&\\
   -1 &2&-1&&&\\
   &-1 &2&-1&&\\
   &&\ddots &\ddots&\ddots&\\
   && &-1 &2&-1\\
   &&&& &h^2
  \end{psmallmatrix},
\label{K-matrix} 
\end{equation}
which described the directed graph illustrated in Figure \ref{fig-5-point-Laplacian-1d}. Let $\mathbf 1$ denote a vector of suitable size with each entry as $1$, then $(K \mathbf 1)_i=\begin{cases}
               0 & i=1,\cdots, n\\
               1 & i=0, n+1                                                                                                                                                                                                \end{cases}.
$ By  Figure \ref{fig-5-point-Laplacian-1d}, it is easy to see that $K$ connects $\mathcal N^0(K\mathbf 1)$ with
$\mathcal N^+(K \mathbf 1)$.

\begin{figure}
\begin{center}
\subfigure[Grid points.]{ 
 \scalebox{1}{
\begin{tikzpicture}[samples=100, domain=-3:3]

\tikzset{vertex/.style = {shape=circle,draw=black,fill=black,thick,inner sep=0pt,minimum size=1.5mm}}
\tikzset{vertex2/.style = {shape=circle,draw=blue!50,fill=blue,thick,
inner sep=0pt,minimum size=1.5mm}}

\tikzset{edge/.style = {->,> = latex'}}

\node[vertex2] (d2) at  (-2,0) {};
\node[vertex2] (e2) at  (2,0) {};
\node[vertex] (a2) at  (-1,0) {};
\node[vertex] (b2) at  (0,0) {};
\node[vertex] (c2) at  (1,0) {};

 \end{tikzpicture}
}} 
\hspace{2cm}
\subfigure[The directed graph.]{ 
 \scalebox{1}{
\begin{tikzpicture}[samples=100, domain=-3:3]

\tikzset{vertex/.style = {shape=circle,draw=black,fill=black,thick,inner sep=0pt,minimum size=1.5mm}}
\tikzset{vertex2/.style = {shape=circle,draw=blue!50,fill=blue,thick,
inner sep=0pt,minimum size=1.5mm}}

\tikzset{edge/.style = {->,> = latex'}}

\node[vertex2] (d2) at  (-2,0) {};
\node[vertex2] (e2) at  (2,0) {};
\node[vertex] (a2) at  (-1,0) {};
\node[vertex] (b2) at  (0,0) {};
\node[vertex] (c2) at  (1,0) {};
\draw[edge] (c2) to (e2);
\draw[edge] (a2) to (d2);
\draw[edge] (b2) to[bend left] (a2);
\draw[edge] (a2) to[bend left] (b2);
\draw[edge] (b2) to[bend left] (c2);
\draw[edge] (c2) to[bend left] (b2);

 \end{tikzpicture}
}} 
\end{center}
\caption{An illustration of the directed graph described by off-diagonal entries of the matrix in  \eqref{K-matrix}: the domain $[0,1]$ is discretized by a uniform $5$-point grid; the black points are interior grid points and the blue ones are the boundary grid points. There is a directed path from any interior grid point to at least one of the boundary points. }

\label{fig-5-point-Laplacian-1d}
\end{figure}

Next we consider the second order accurate 5-point discrete Laplacian scheme for solving $-\Delta u=f$ on $\Omega=(0,1)\times(0,1)$ with homogeneous Dirichlet boundary conditions:
 \[u_{i,j}=0, (x_i, y_j)\in \partial \Omega;  
 \frac{-u_{i-1,j}-u_{i+1,j}+4u_{i,j}-u_{i,j+1}-u_{i+1,j}}{h^2}=f_{ij}, (x_i, y_j)\in \Omega.\]
See Figure \ref{fig-5-point-Laplacian} for the directed graph described by its matrix representation.  Let $K$ be the matrix representation of the 5-point discrete Laplacian scheme, then 
$$(K \mathbf 1)_{i,j}=\begin{cases}
                                1, & \mbox{if $(x_i, y_j)\in \partial \Omega$},\\
                                0, &  \mbox{if $(x_i, y_j)\in  \Omega$}.
                               \end{cases}
$$
By  Figure \ref{fig-5-point-Laplacian}, it is easy to see that $K$ connects $\mathcal N^0(K \mathbf 1)$ with
$\mathcal N^+(K \mathbf 1)$. 

\begin{figure}
\begin{center}
\subfigure[Grid points.]{ 
 \scalebox{1}{
\begin{tikzpicture}[samples=100, domain=-3:3]

\tikzset{vertex/.style = {shape=circle,draw=black,fill=black,thick,inner sep=0pt,minimum size=1.5mm}}
\tikzset{vertex2/.style = {shape=circle,draw=blue!50,fill=blue,thick,
inner sep=0pt,minimum size=1.5mm}}

\tikzset{edge/.style = {->,> = latex'}}

\node[vertex2] (a) at  (-2,-2) {};
\node[vertex2] (a) at  (-2,2) {};
\node[vertex2] (e) at  (2,-2) {};

\node[vertex2] (a) at  (2,2) {};

\node[vertex2] (d3) at  (-2,-1) {};
\node[vertex2] (d2) at  (-2,0) {};
\node[vertex2] (d1) at  (-2,1) {};
\node[vertex2] (e3) at  (2,-1) {};
\node[vertex2] (e2) at  (2,0) {};
\node[vertex2] (e1) at  (2,1) {};
\node[vertex2] (f1) at  (-1,2) {};
\node[vertex2] (f2) at  (0,2) {};
\node[vertex2] (f3) at  (1,2) {};
\node[vertex2] (g1) at  (-1,-2) {};
\node[vertex2] (g2) at  (0,-2) {};
\node[vertex2] (g3) at  (1,-2) {};

\node[vertex] (a1) at  (-1,1) {};
\node[vertex] (b1) at  (0,1) {};
\node[vertex] (c1) at  (1,1) {};
\node[vertex] (a2) at  (-1,0) {};
\node[vertex] (b2) at  (0,0) {};
\node[vertex] (c2) at  (1,0) {};
\node[vertex] (a3) at  (-1,-1) {};
\node[vertex] (b3) at  (0,-1) {};
\node[vertex] (c3) at  (1,-1) {};

 \end{tikzpicture}
}} \hspace{.6in}
\subfigure[The directed graph.]{
 \scalebox{1}{
\begin{tikzpicture}[samples=100, domain=-3:3]

\tikzset{vertex/.style = {shape=circle,draw=black,fill=black,thick,inner sep=0pt,minimum size=1.5mm}}
\tikzset{vertex2/.style = {shape=circle,draw=blue!50,fill=blue,thick,
inner sep=0pt,minimum size=1.5mm}}

\tikzset{edge/.style = {->,> = latex'}}

\node[vertex2] (a) at  (-2,-2) {};
\node[vertex2] (a) at  (-2,2) {};
\node[vertex2] (e) at  (2,-2) {};

\node[vertex2] (a) at  (2,2) {};

\node[vertex2] (d3) at  (-2,-1) {};
\node[vertex2] (d2) at  (-2,0) {};
\node[vertex2] (d1) at  (-2,1) {};
\node[vertex2] (e3) at  (2,-1) {};
\node[vertex2] (e2) at  (2,0) {};
\node[vertex2] (e1) at  (2,1) {};
\node[vertex2] (f1) at  (-1,2) {};
\node[vertex2] (f2) at  (0,2) {};
\node[vertex2] (f3) at  (1,2) {};
\node[vertex2] (g1) at  (-1,-2) {};
\node[vertex2] (g2) at  (0,-2) {};
\node[vertex2] (g3) at  (1,-2) {};

\node[vertex] (a1) at  (-1,1) {};
\node[vertex] (b1) at  (0,1) {};
\node[vertex] (c1) at  (1,1) {};
\node[vertex] (a2) at  (-1,0) {};
\node[vertex] (b2) at  (0,0) {};
\node[vertex] (c2) at  (1,0) {};
\node[vertex] (a3) at  (-1,-1) {};
\node[vertex] (b3) at  (0,-1) {};
\node[vertex] (c3) at  (1,-1) {};

\draw[edge] (b1) to[bend left] (a1);
\draw[edge] (a1) to[bend left] (b1);
\draw[edge] (b1) to[bend left] (c1);
\draw[edge] (c1) to[bend left] (b1);
\draw[edge] (b2) to[bend left] (a2);
\draw[edge] (a2) to[bend left] (b2);
\draw[edge] (b2) to[bend left] (c2);
\draw[edge] (c2) to[bend left] (b2);
\draw[edge] (b3) to[bend left] (a3);
\draw[edge] (a3) to[bend left] (b3);
\draw[edge] (b3) to[bend left] (c3);
\draw[edge] (c3) to[bend left] (b3);

\draw[edge] (a2) to[bend left] (a1);
\draw[edge] (a1) to[bend left] (a2);
\draw[edge] (b2) to[bend left] (b1);
\draw[edge] (b1) to[bend left] (b2);
\draw[edge] (c2) to[bend left] (c1);
\draw[edge] (c1) to[bend left] (c2);
\draw[edge] (a2) to[bend left] (a3);
\draw[edge] (a3) to[bend left] (a2);
\draw[edge] (b2) to[bend left] (b3);
\draw[edge] (b3) to[bend left] (b2);
\draw[edge] (c2) to[bend left] (c3);
\draw[edge] (c3) to[bend left] (c2);

\draw[edge] (a1) to (f1);
\draw[edge] (b1) to (f2);
\draw[edge] (c1) to (f3);
\draw[edge] (a1) to (d1);
\draw[edge] (a2) to (d2);
\draw[edge] (a3) to (d3);
\draw[edge] (c1) to (e1);
\draw[edge] (c2) to (e2);
\draw[edge] (c3) to (e3);
\draw[edge] (a3) to (g1);
\draw[edge] (b3) to (g2);
\draw[edge] (c3) to (g3);
 \end{tikzpicture}
}
}
\end{center}
\caption{An illustration of the directed graph described by off-diagonal entries in the 5-point discrete Laplacian matrix: the domain $[0,1]\times[0,1]$ is discretized by a uniform $5\times 5$ grid; the black points are interior grid points and the blue ones are the boundary grid points. There is a directed path from any interior grid point to at least one of the boundary grid points. }

\label{fig-5-point-Laplacian}
\end{figure}

Let $A:=\bar{L}_h$ denote the matrix representation of any scheme in Section \ref{sec-q2fdscheme}. 
Then $$(A \mathbf 1)_{i,j}=\begin{cases}
                                1, & \mbox{if $(x_i, y_j)\in \partial \Omega$},\\
                                c_{ij}\geq 0, &  \mbox{if $(x_i, y_j)\in  \Omega$}.
                               \end{cases}
$$
Therefore,
$\mathcal N^+{(K\mathbf 1)}\subset \mathcal N^+{(A\mathbf 1)}$ 
implies $\mathcal N^0{(A\mathbf 1)}\subset \mathcal N^0{(K\mathbf 1)}$, thus $K$ also
connects $\mathcal N^0(A \mathbf 1)$ with $\mathcal N^+(A \mathbf 1)$.
Notice that indices of nonzero off-diagonal entries in $K$ is a subset of indices of nonzero entries in $A_a^-$, thus 
$A_a^-$ also connects $\mathcal N^0(A \mathbf 1)$ with $\mathcal N^+(A \mathbf 1)$.
So the vector $\mathbf e$ can be set as $\mathbf 1$ in \eqref{cond3}.
If assuming $c(x,y)>0$, then $A\mathbf 1>  0$ thus the condition \eqref{cond3} is trivially satisfied. 

By Theorem \ref{rowsumcondition-thm}, for any decomposition of off-diagonal negative entries  $A^-_a = A^z + A^s$, $A_d+A^z$ is an M-matrix if $(A_d+A^z)\mathbf 1\neq \mathbf 0$ and $(A_d+A^z)\mathbf 1\geq 0$.
So Theorem \ref{thm3} for the schemes \eqref{p2fd-vcoef-1dscheme} and \eqref{p2fd-vcoef-2dscheme} can be simplified as 
\begin{theorem}
 \label{newthm3}
Let $A$ denote the matrix representation of the schemes solving $-\nabla\cdot( a\nabla)u+cu=f$ in Section \ref{sec-q2fdscheme}. 
Assume $A^-_a$ has a decomposition $A^-_a = A^z + A^s$ with $A^s\leq 0$ and $A^z \leq 0$.  
Then $A^{-1}\geq 0$ if the following are satisfied:
\begin{enumerate}
\item $(A_d+A^z)\mathbf 1\neq \mathbf 0$  and $(A_d+A^z)\mathbf 1\geq 0$;
\item $A^+_a \leq A^zA^{-1}_dA^s$;
 \item For $c(x,y)\geq 0$, either $A^z$ or $A^s$ has the same sparsity pattern as $A^-_a$. 
 If $c(x,y)>0$, then this condition can be removed.  
\end{enumerate}

\end{theorem}

\subsection{One-dimensional Laplacian case}
As a demonstration of how to apply Theorem \ref{newthm3}, we first consider  the scheme \eqref{p2fd-laplacian-1dscheme}.
Let $A$ be the matrix representation of the linear operator $\mathcal L_h$ in the scheme \eqref{p2fd-laplacian-1dscheme}. Let $\mathcal A_d$ and $\mathcal A_a^{\pm}$ be linear operators corresponding to the matrices $A_d$ and $A_a^{\pm}$ respectively. 

Consider the following decomposition of $\mathcal A_a^-=\mathcal A^z+\mathcal A^s$
with $\mathcal A^z=\mathcal A^s=\frac12 \mathcal A_a^-$:
 \begin{align*}
  \mathcal A^z(\bar{\mathbf u})_0=\mathcal A^s(\bar{\mathbf u})_0&=0, \quad 
  \mathcal A^z (\bar{\mathbf u})_{n+1}=\mathcal A^s (\bar{\mathbf u})_{n+1}=0, \\
  \mathcal A^z(\bar{\mathbf u})_i=\mathcal A^s(\bar{\mathbf u})_i&=\frac{-u_{i-1}-u_{i+1}}{2h^2}, \quad \mbox{if $x_i$ is a cell center},\\
   \mathcal A^z(\bar{\mathbf u})_i=\mathcal A^s(\bar{\mathbf u})_i&=
\frac{-8u_{i-1}-8u_{i+1}}{8h^2},\quad \mbox{if $x_i$ is an interior cell end}.
  \end{align*}
The operator $\mathcal A_d$ and $\mathcal A^+_a$ are given as:  
 \begin{align*}
\mathcal A_d(\bar{\mathbf u})_0&=u_0, \quad 
 \mathcal A_d (\bar{\mathbf u})_{n+1}=u_{n+1}, \\
  \mathcal A_d(\bar{\mathbf u})_i&=\frac{2u_i}{h^2}, \quad \mbox{if $x_i$ is a cell center},\\
   \mathcal A_d(\bar{\mathbf u})_i&=
\frac{14u_i}{4h^2},\quad \mbox{if $x_i$ is an interior cell end}.
  \end{align*}
   \begin{align*}
  \mathcal A^+_a(\bar{\mathbf u})_0&=0, \quad 
  \mathcal A^+_a (\bar{\mathbf u})_{n+1}=0, \\
  \mathcal A^+_a(\bar{\mathbf u})_i&=0, \quad \mbox{if $x_i$ is a cell center},\\
   \mathcal A^+_a(\bar{\mathbf u})_i&=\frac{u_{i-2}+u_{i+2}}{4h^2},\quad \mbox{if $x_i$ is an interior cell end}.
  \end{align*}
  Obviously, $A^z$ and $A^s$ both have have the same sparsity pattern as $A^-_a$. It is straightforward to verify $[\mathcal A^d+\mathcal A^z](\mathbf 1)$ is a non-negative nonzero vector. So we only need to verify  $A^+_a \leq A^zA^{-1}_dA^s$ to apply Theorem \ref{newthm3}.
  Since $A^zA^{-1}_dA^s\geq 0$, we only need to compare nonzero coefficients in 
  $\mathcal A^+_a(\bar{\mathbf u})_i$ and $\mathcal A^z\left(\mathcal A^{-1}_d[\mathcal A^s(\bar{\mathbf u})]\right)_i$.
  
  When $x_{i}$ is an interior cell end, $x_{i\pm1}$ are cell centers, and we have
  \[\mathcal A^s(\bar{\mathbf u})_{i-1}=\frac{-u_{i-2}-u_i}{2h^2},\quad \mathcal A^{-1}_d[\mathcal A^s(\bar{\mathbf u})]_{i-1}=\frac{h^2\mathcal A^s(\bar{\mathbf u})_{i-2}}{2},\]
  \[ \mathcal A^z(\mathcal A^{-1}_d[\mathcal A^s(\bar{\mathbf u})])_i=\frac{-\mathcal A^{-1}_d[-\mathcal A^s(\bar{\mathbf u})]_{i-1}-\mathcal A^{-1}_d[\mathcal A^s(\bar{\mathbf u})]_{i+1}}{h^2}=\frac{u_{i-2}+2u_i+u_{i+2}}{4h^2}.\]
We can verify  $A^+_a \leq A^zA^{-1}_dA^s$  by comparing only the coefficients of $u_{i\pm2}$ in  $\mathcal A^+_a(\bar{\mathbf u})_i$ and $\mathcal A^z\left(\mathcal A^{-1}_d[\mathcal A^s(\bar{\mathbf u})]\right)_i$ because $A^z A^{-1}_d A^s\geq 0$. 
By Theorem \ref{newthm3}, we get $A^{-1}\geq 0.$

\subsection{One-dimensional variable coefficient case}
As we have seen in the previous discussion, all the operators are either zero or identity at the boundary points thus do not affect the discussion verifying the condition \eqref{cond2}. For the sake of simplicity, we only consider the interior grid points for the linear operators. 
With the positive and negative parts for a number $f$ defined as: 
\[ f^+=\frac{|f|+f}{2},\quad f^-=\frac{|f|-f}{2},\]  
the linear operators $\mathcal A_d$, $\mathcal A_a^{\pm}$ are
 \begin{align*}
  \mathcal A_d(\bar{\mathbf u})_i&=\left(\frac{a_{i-1}+a_{i+1}}{h^2}+c_i\right) u_i, \quad \mbox{if $x_i$ is a cell center},\\
   \mathcal A_d(\bar{\mathbf u})_i&=
\left(\frac{a_{i-2}+4a_{i-1}+18 a_i+4a_{i+1}+a_{i+2}}{8h^2}+c_i \right)u_i,\quad \mbox{if $x_i$ is an interior cell end}.
  \end{align*}
 \begin{align*}
  \mathcal A_a^+(\bar{\mathbf u})_i&=0, \quad \mbox{if $x_i$ is a cell center},\\
   \mathcal A_a^+(\bar{\mathbf u})_i&=
\frac{(3a_{i-2}-4a_{i-1}+3a_i)^+u_{i-2}+(3a_{i+2}-4a_{i+1}+3 a_i)^+u_{i+2}}{8h^2},\quad \mbox{if $x_i$ is an interior cell end}.
  \end{align*}
       \begin{align*}
  & \mbox{If $x_i$ is a cell center}\quad \mathcal A_a^-(\bar{\mathbf u})_i=\frac{-(3a_{i-1}+a_{i+1})u_{i-1}-(a_{i-1}+3a_{i+1})u_{i+1}}{4h^2}, \\
  &\mbox{If $x_i$ is an interior cell end}\quad  \mathcal A_a^-(\bar{\mathbf u})_i=\\
  &
\frac{-(3a_{i-2}-4a_{i-1}+3a_i)^-u_{i-2}-(4a_{i-2}+12a_i)u_{i-1}-(12a_{i}+4a_{i+2})u_{i+1}-(3a_{i}-4a_{i+1}+3 a_{i+2})^-u_{i+2}}{8h^2}.
  \end{align*}

  We can easily verify that $(A_d+A^z)\mathbf 1\geq 0$ for the following $\mathcal A^z$:
   \begin{align*}
  & \mbox{if $x_i$ is a cell center}\quad \mathcal A^z(\bar{\mathbf u})_i=\epsilon\frac{-(3a_{i-1}+a_{i+1})u_{i-1}-(a_{i-1}+3a_{i+1})u_{i+1}}{4h^2}, \\
  &\mbox{if $x_i$ is an interior cell end}\quad  \mathcal A^z(\bar{\mathbf u})_i=\\
  &
\frac{-(3a_{i-2}-4a_{i-1}+3a_i)^-u_{i-2}-[4a_{i-2}+12a_i-(3a_{i-2}-4a_{i-1}+3a_i)^+]u_{i-1}}{8h^2}\\
&+\frac{-[12a_{i}+4a_{i+2}-(3a_{i}-4a_{i+1}+3 a_{i+2})^+]u_{i+1}-(3a_{i}-4a_{i+1}+3 a_{i+2})^-u_{i+2}}{8h^2},
  \end{align*}
  where $\epsilon>0$ is a small number.  
Moreover, $A^z$ has the same sparsity pattern as $A^-_a$ for any $\epsilon>0$. For $\epsilon<1$ we can verify that $A^s=A^-_a-A^z\leq 0$:
     \begin{align*}
  & \mbox{If $x_i$ is a cell center}\quad \mathcal A^s(\bar{\mathbf u})_i=(1-\epsilon)\frac{-(3a_{i-1}+a_{i+1})u_{i-1}-(a_{i-1}+3a_{i+1})u_{i+1}}{4h^2}, \\
  &\mbox{If $x_i$ is an interior cell end}\quad  \mathcal A^s(\bar{\mathbf u})_i=\frac{-(3a_{i-2}-4a_{i-1}+3a_i)^+u_{i-1}-(3a_{i}-4a_{i+1}+3 a_{i+2})^+ u_{i+1}}{8h^2}.
  \end{align*}
Now
 we only need to compare nonzero coefficients in 
  $\mathcal A^+_a(\bar{\mathbf u})_i$ and $\mathcal A^z\left(\mathcal A^{-1}_d[\mathcal A^s(\bar{\mathbf u})]\right)_i$ for $x_i$ being an interior cell end.
    When $x_{i}$ is an interior cell end, $x_{i\pm1}$ are cell centers, and we have
  \[\mathcal A^s(\bar{\mathbf u})_{i-1}=(1-\epsilon)\frac{-(3a_{i-2}+a_{i})u_{i-2}-(a_{i-2}+3a_{i})u_{i}}{4h^2},\]
  \[ A^s(\bar{\mathbf u})_{i-2}=\frac{-(3a_{i-4}-4a_{i-3}+3a_{i-2})^+u_{i-3}-(3a_{i-2}-4a_{i-1}+3 a_{i})^+ u_{i-1}}{8h^2}\]
  \[\mathcal A^{-1}_d[\mathcal A^s(\bar{\mathbf u})]_{i-1}=\frac{h^2 }{(a_{i-2}+a_i+h^2 c_{i-1})}\mathcal A^s(\bar{\mathbf u})_{i-1}=(1-\epsilon)\frac{-(3a_{i-2}+a_{i})u_{i-2}-(a_{i-2}+3a_{i})u_{i}}{4(a_{i-2}+a_i+h^2 c_{i-1})}.\]
It suffices to focus on the coefficient of $u_{i-2}$ in  $\mathcal A^z(\mathcal A^{-1}_d[\mathcal A^s(\bar{\mathbf u})])_i$ and the discussion for the coefficient of $u_{i+2}$ is similar. Notice that $\mathcal A^{-1}_d[\mathcal A^s(\bar{\mathbf u})]_{i-2}$ will contribute nothing to the coefficient of $u_{i-2}$. So the coefficient of $u_{i-2}$ in  $\mathcal A^z(\mathcal A^{-1}_d[\mathcal A^s(\bar{\mathbf u})])_i$ is 
\[ (1-\epsilon)\frac{(3a_{i-2}+a_{i})(4a_{i-2}+12a_i-(3a_{i-2}-4a_{i-1}+3a_i)^+)}{32h^2(a_{i-2}+a_i+h^2 c_{i-1})}.\]
Thus to ensure $A^+_a\leq A^zA_d^-A^s$, it suffices to have the following holds for any interior cell end $x_i$:
\begin{equation*}(1-\epsilon)\frac{(3a_{i-2}+a_{i})(4a_{i-2}+12a_i-(3a_{i-2}-4a_{i-1}+3a_i)^+)}{32h^2(a_{i-2}+a_i+h^2 c_{i-1})}\geq \frac{(3a_{i-2}-4a_{i-1}+3a_i)^+}{8h^2}.
\end{equation*}
Equivalently, we need the following inequality holds for any cell center $x_i$: 
\begin{equation}(1-\epsilon)\frac{(3a_{i-1}+a_{i+1})(4a_{i-1}+12a_{i+1}-(3a_{i-1}-4a_{i}+3a_{i+1})^+)}{32h^2(a_{i-1}+a_{i+1}+h^2 c_{i})}\geq \frac{(3a_{i-1}-4a_{i}+3a_{i+1})^+}{8h^2}.\label{h-condition-1}
\end{equation}

Notice that $\epsilon$ can be any fixed number in $[0,1)$ so that $A_d+A^z$ is an M-matrix and $A^s\leq 0$. And $\epsilon$ must be strictly positive so that $A^z$ has the same sparsity pattern as $A_a^-$. 
Thus if there is one fixed $\epsilon\in(0,1)$ so that \eqref{h-condition-1} holds for any cell center $x_{i}$, then by Theorem \ref{newthm3}, $A^{-1}\geq 0.$
A sufficient condition for \eqref{h-condition-1} to hold for any cell center $x_i$ with some fixed $\epsilon\in(0,1)$ is to have the following inequality for any cell center $x_i$:
\begin{equation}\frac{(3a_{i-1}+a_{i+1})(4a_{i-1}+12a_{i+1}-(3a_{i-1}-4a_{i}+3a_{i+1})^+)}{32h^2(a_{i-1}+a_{i+1}+h^2 c_{i})}>\frac{(3a_{i-1}-4a_{i}+3a_{i+1})^+}{8h^2}.\label{h-condition-simplified}
\end{equation}

If $3a_{i-1}-4a_{i}+3a_{i+1}\leq 0$, then \eqref{h-condition-simplified} holds trivially. 
We only need to discuss the case $3a_{i-1}-4a_{i}+3a_{i+1}>0$, for which \eqref{h-condition-simplified} becomes
\begin{equation}(3a_{i-1}+a_{i+1})(a_{i-1}+4a_{i}+9a_{i+1})> 4(a_{i-1}+a_{i+1}+h^2 c_{i})(3a_{i-1}-4a_{i}+3a_{i+1}).
\label{h-condition-2}
\end{equation}

So we have proven the first result for the variable coefficient case:
\begin{theorem}
 For the scheme \eqref{p2fd-vcoef-1dscheme} solving  $- (a u')'+cu=f$ with $a(x)>0$ and $c(x)\geq 0$, its matrix representation $A=\bar{L}_h$ satisfies $A^{-1}\geq 0$ if  
 \eqref{h-condition-2} holds for any cell center $x_i$.
\end{theorem}
The constraint \eqref{h-condition-2} will be satisfied for small enough $h$. The proof of the following two theorems are included in the Appendix \ref{appendix-b}. 
\begin{theorem}
\label{1d-thm-mesh-1}
  For the scheme \eqref{p2fd-vcoef-1dscheme} solving  $- (a u')'+cu=f$ with $a(x)>0$ and $c(x)\geq 0$ on a uniform mesh, its matrix representation $A=\bar{L}_h$ satisfies $A^{-1}\geq 0$ if any of  the following constraints 
  is satisfied for each finite element cell $I_i=[x_{i-1}, x_{i+1}]$:
  \begin{itemize}
   \item There exists some $\lambda\in(\frac{3}{13},1)$ such that 
   \[h^2 c_i< \frac{13(1-\lambda)\min\limits_{I_i} a^2(x)}{6\max\limits_{I_i} a(x)-4\min\limits_{I_i} a(x)},\qquad h\frac{\max\limits_{x\in I_i} \left|a'(x)\right |}{\min\limits_{x\in I_i}a(x)}<\frac{\sqrt{39\lambda}-3}{6}.\]
\item   
 $ 2 h \max\limits_{I_i} |a'(x)|+h^2 c_i \left (1-\frac23\frac{\min\limits_{I_i} a(x)}{\max\limits_{I_i} a(x)}\right)< \frac{5}{3} \frac{\min\limits_{I_i} a^2(x)}{\max\limits_{I_i} a(x)}.$
  \item  If $c(x)\equiv0$, then we only need
    $ h\frac{\max\limits_{x\in I_i} \left|a'(x)\right |}{\min\limits_{x\in I_i}a(x)}<\frac{\sqrt{39}-3}{6}.$
 \item If $a(x)\equiv a>0$, then we only need 
    $h^2 c_i< 5 a.$
  \end{itemize}
\end{theorem}

\begin{theorem}
\label{1d-thm-mesh-2}
For the scheme \eqref{p2fd-vcoef-1dscheme} solving  $- (a u')'+cu=f$  with $a(x)>0$ and $c(x)\geq 0$, its matrix representation $A=\bar{L}_h$ satisfies $A^{-1}\geq 0$ if  
the following mesh constraint is achieved for all cell centers $x_i$:
\begin{subequations}
\label{mesh-constraint-1}
 \begin{equation}
 \label{mesh-constraint-1a}
 h^2\left(\frac32 c_i+\max_{x\in(x_{i-1}, x_{i+1})} a''(x)\right) < \frac{74}{45} \min\{a_{i-1},a_i,a_{i+1}\}.
 \end{equation}
If $a(x)$ is a concave function, then \eqref{mesh-constraint-1a} can be replaced by 
 \begin{equation}
 h^2 c_i <3 \min\{a_{i-1},a_i,a_{i+1}\}.
 \end{equation}
\end{subequations}

\end{theorem}

\begin{rmk}
 For solving heat equation with backward Euler time discretization \eqref{backwardeuler}, the mesh constraints in Theorem \ref{1d-thm-mesh-1} and Theorem \ref{1d-thm-mesh-2} imply that a lower bound for $\frac{\Delta t}{h^2}$ is a sufficient condition for ensuring monotonicity. Numerical tests suggest that a lower bound on $\frac{\Delta t}{h^2}$ is also a necessary condition, see Section \ref{sec-test}. 
A lower bound constraint on the time step is common for high order accurate spatial discretizations with backward Euler to satisfy monotonicity, e.g., \cite{qin2018implicit}. 
\end{rmk}


\subsection{Two-dimensional variable coefficient case}
Next we apply Theorem \ref{newthm3} to the scheme \eqref{p2fd-vcoef-2dscheme}. 
The splitting $A_a^-=A^z+A^s$ is quite similar to one-dimensional case due to its stencil pattern.

 Let $A:=\bar{L}_h$ be the matrix representation of the linear operator $\mathcal L_h$ in the scheme  \eqref{p2fd-vcoef-2dscheme}. We only consider interior grid points since $\mathcal L_h$ is identity operator on boundary points which do  not affect  applying Theorem \ref{newthm3}. We first have
 \begin{align*}
&   \mathcal A_d (\bar{\mathbf u})_{ij}=\left(\frac{a_{i-1,j}+a_{i+1,j}+a_{i,j-1}+a_{i,j+1}}{h^2}+c_{ij} \right)u_{ij},\quad \mbox{if $x_{ij}$ is a cell center};\\
&   \mathcal A_d (\bar{\mathbf u})_{ij}=\left( \frac{(a_{i-2,j}+4a_{i-1,j}+18 a_{ij}+4a_{i+1,j}+a_{i+2,j})+8(a_{i,j-1}+a_{i,j+1})}{8h^2}+c_{ij}\right)u_{ij},\\   
& \mbox{if $x_{ij}$ is an edge center for an edge parallel to $y$-axis};\\
&   \mathcal A_d (\bar{\mathbf u})_{ij}=\left( \frac{(a_{i,j-2}+4a_{i,j-1}+18 a_{ij}+4a_{i,j+1}+a_{i,j+2})+8(a_{i-1,j}+a_{i+1,j})}{8h^2}+c_{ij}\right)u_{ij},\\   
& \mbox{if $x_{ij}$ is an edge center for an edge parallel to $x$-axis};\\
&   \mathcal A_d (\bar{\mathbf u})_{ij}=\left( \frac{(a_{i-2,j}+4a_{i-1,j}+18 a_{ij}+4a_{i+1,j}+a_{i+2,j})+(a_{i,j-2}+4a_{i,j-1}+18 a_{ij}+4a_{i,j+1}+a_{i,j+2})}{8h^2}+c_{ij}\right)u_{ij},\\   
& \mbox{if $x_{ij}$ is a knot}.
\end{align*}
   
For the operator $\mathcal A^+_a$, it is given as
 \begin{align*}
&   \mathcal A_a^+ (\bar{\mathbf u})_{ij}=0,\quad \mbox{if $x_{ij}$ is a cell center};\\
&   \mathcal A_a^+ (\bar{\mathbf u})_{ij}=\frac{(3a_{i-2,j}-4a_{i-1,j}+3a_{i,j})^+u_{i-2,j}+(3a_{i+2,j}-4a_{i+1,j}+3 a_{i,j})^+u_{i+2,j}}{8h^2}\\
& \mbox{if $x_{ij}$ is an edge center for an edge parallel to $y$-axis};\\
&   \mathcal A_a^+ (\bar{\mathbf u})_{ij}=\frac{(3a_{i,j-2}-4a_{i,j-1}+3a_{i,j})^+u_{i,j-2}+(3a_{i,j+2}-4a_{i,j+1}+3 a_{i,j})^+u_{i,j+2}}{8h^2}\\
& \mbox{if $x_{ij}$ is an edge center for an edge parallel to $x$-axis};\\
&   \mathcal A_a^+ (\bar{\mathbf u})_{ij}=\frac{(3a_{i-2,j}-4a_{i-1,j}+3a_{i,j})^+u_{i-2,j}+(3a_{i+2,j}-4a_{i+1,j}+3 a_{i,j})^+u_{i+2,j}}{8h^2}\\
&+\frac{(3a_{i,j-2}-4a_{i,j-1}+3a_{i,j})^+u_{i,j-2}+(3a_{i,j+2}-4a_{i,j+1}+3 a_{i,j})^+u_{i,j+2}}{8h^2}\\
& \mbox{if $x_{ij}$ is a knot}.
\end{align*}

Let $\epsilon\in(0,1)$ be a fixed number. We consider the following $A^z\leq 0$ so that $(A_d+A^z)\mathbf 1\geq 0$:
    \begin{align*} 
    & \mbox{if $x_{ij}$ is a cell center} \quad \mathcal A^z (\bar{\mathbf u})_{ij}=-\epsilon\frac{(3a_{i-1,j}+a_{i+1,j})u_{i-1,j}+(a_{i-1,j}+3a_{i+1,j})u_{i+1,j}}{4h^2}\\
-&\epsilon\frac{(3a_{i,j-1}+a_{i,j+1})u_{i,j-1}+(a_{i,j-1}+3a_{i,j+1})u_{i,j+1}}{4h^2};\\
& \mbox{if $x_{ij}$ is an edge center for an edge parallel to $y$-axis}, \mathcal A^z (\bar{\mathbf u})_{ij}= \\
\end{align*}
\begin{align*}
& \frac{-(3a_{i-2,j}-4a_{i-1,j}+3a_{i,j})^-u_{i-2,j}-[4a_{i-2,j}+12a_{i,j}-(3a_{i-2,j}-4a_{i-1,j}+3a_{i,j})^+]u_{i-1,j}}{8h^2} \\
&+\frac{-[12a_{i,j}+4a_{i+2,j}-(3a_{i+2,j}-4a_{i+1,j}+3 a_{i,j})^+]u_{i+1,j}-(3a_{i+2,j}-4a_{i+1,j}+3 a_{i,j})^-u_{i+2,j}}{8h^2}\\
& +\epsilon\frac{-(3a_{i,j-1}+a_{i,j+1})u_{i,j-1}-(a_{i,j-1}+3a_{i,j+1})u_{i,j+1}}{4h^2};\\
\end{align*}
\begin{align*} 
& \mbox{if $x_{ij}$ is an edge center for an edge parallel to $x$-axis}, \mathcal A^z (\bar{\mathbf u})_{ij}= \\
& \frac{-(3a_{i,j-2}-4a_{i,j-1}+3a_{i,j})^-u_{i,j-2}-[4a_{i,j-2}+12a_{i,j}-(3a_{i,j-2}-4a_{i,j-1}+3a_{i,j})^+]u_{i,j-1}}{8h^2} \\
&+\frac{-[12a_{i,j}+4a_{i,j+2}-(3a_{i,j+2}-4a_{i,j+1}+3 a_{i,j})^+]u_{i,j+1}-(3a_{i,j+2}-4a_{i,j+1}+3 a_{i,j})^-u_{i,j+2}}{8h^2}\\\end{align*}
\begin{align*} 
& +\epsilon\frac{-(3a_{i-1,j}+a_{i+1,j})u_{i-1,j}-(a_{i-1,j}+3a_{i+1,j})u_{i+1,j}}{4h^2};\\
& \mbox{if $x_{ij}$ is a knot},\quad\mathcal A^z (\bar{\mathbf u})_{ij}= \\
& \frac{-(3a_{i-2,j}-4a_{i-1,j}+3a_{i,j})^-u_{i-2,j}-[4a_{i-2,j}+12a_{i,j}-(3a_{i-2,j}-4a_{i-1,j}+3a_{i,j})^+]u_{i-1,j}}{8h^2} \\\end{align*}
\begin{align*} 
&+\frac{-[12a_{i,j}+4a_{i+2,j}-(3a_{i+2,j}-4a_{i+1,j}+3 a_{i,j})^+]u_{i+1,j}-(3a_{i+2,j}-4a_{i+1,j}+3 a_{i,j})^-u_{i+2,j}}{8h^2};\\\end{align*}
\begin{align*} 
&+\frac{-(3a_{i,j-2}-4a_{i,j-1}+3a_{i,j})^-u_{i,j-2}-[4a_{i,j-2}+12a_{i,j}-(3a_{i,j-2}-4a_{i,j-1}+3a_{i,j})^+]u_{i,j-1}}{8h^2} \\
&+\frac{-[12a_{i,j}+4a_{i,j+2}-(3a_{i,j+2}-4a_{i,j+1}+3 a_{i,j})^+]u_{i,j+1}-(3a_{i,j+2}-4a_{i,j+1}+3 a_{i,j})^-u_{i,j+2}}{8h^2};\\
\end{align*}

 Then  $A^s=A_a^- -A^z$ is given as:
\begin{align*} 
&\mbox{if $x_i$ is a cell center}, \quad \mathcal A^s (\bar{\mathbf u})_{ij}=-(1-\epsilon)\frac{(3a_{i-1,j}+a_{i+1,j})u_{i-1,j}+(a_{i-1,j}+3a_{i+1,j})u_{i+1,j}}{4h^2}\\
-&(1-\epsilon)\frac{(3a_{i,j-1}+a_{i,j+1})u_{i,j-1}+(a_{i,j-1}+3a_{i,j+1})u_{i,j+1}}{4h^2};\\
& \mbox{if $x_{ij}$ is an edge center for an edge parallel to $y$-axis}, \\
& \mathcal A^s (\bar{\mathbf u})_{ij}=  \frac{-(3a_{i-2,j}-4a_{i-1,j}+3a_{i,j})^+u_{i-1,j}-(3a_{i+2,j}-4a_{i+1,j}+3 a_{i,j})^+u_{i+1,j}}{8h^2}\\
&+(1-\epsilon)\frac{-(3a_{i,j-1}+a_{i,j+1})u_{i,j-1}-(a_{i,j-1}+3a_{i,j+1})u_{i,j+1}}{4h^2};\\\end{align*}
\begin{align*} 
& \mbox{if $x_{ij}$ is an edge center for an edge parallel to $x$-axis},\\
& \mathcal A^s (\bar{\mathbf u})_{ij}=\frac{-(3a_{i,j-2}-4a_{i,j-1}+3a_{i,j})^+u_{i,j-1}-(3a_{i,j+2}-4a_{i,j+1}+3 a_{i,j})^+u_{i,j+1}}{8h^2}\\
&+(1-\epsilon)\frac{-(3a_{i-1,j}+a_{i+1,j})u_{i-1,j}-(a_{i-1,j}+3a_{i+1,j})u_{i+1,j}}{4h^2};\\
\end{align*}
\begin{align*}
& \mbox{if $x_{ij}$ is a knot},\quad\mathcal A^s (\bar{\mathbf u})_{ij}= \frac{-(3a_{i-2,j}-4a_{i-1,j}+3a_{i,j})^+u_{i-1,j}-(3a_{i+2,j}-4a_{i+1,j}+3 a_{i,j})^+u_{i+1,j}}{8h^2} \\
&+\frac{-(3a_{i,j-2}-4a_{i,j-1}+3a_{i,j})^+u_{i,j-1}-(3a_{i,j+2}-4a_{i,j+1}+3 a_{i,j})^+u_{i,j+1}}{8h^2};\\
\end{align*}

For the positive off-diagonal entries, $\mathcal A^+_a (\bar{\mathbf u})_{ij}$ is nonzero only for $x_{ij}$ being an edge center or a cell center.
Thus to verify  $A_a^{+}\leq A^zA_d^{-1}A^s$, it suffices to compare $\mathcal A^z\left[\mathcal A_d^{-1}\left(\mathcal A^s(\bar{\mathbf u})\right)\right]_{ij}$ with 
$\mathcal A_a^{+}(\bar{\mathbf u})_{ij}$ for $x_{ij}$ being an edge center or a cell center.

If $x_{ij}$ is an edge center for an edge parallel to $y$-axis, then $x_{i\pm1,j}$ are cell centers. Since everything here has a symmetric structure, we only need to compare the coefficients of
$u_{i-2,j}$ in $\mathcal A^z\left[\mathcal A_d^{-1}\left(\mathcal A^s(\bar{\mathbf u})\right)\right]_{ij}$ and
$\mathcal A_a^{+}(\bar{\mathbf u})_{ij}$, and the comparison for the coefficients of
$u_{i+2,j}$ will be similar. 
\begin{align*} 
\mathcal A^s (\bar{\mathbf u})_{i-1,j}& =-(1-\epsilon)\frac{(3a_{i-2,j}+a_{ij})u_{i-2,j}+(a_{i-2,j}+3a_{i,j})u_{i,j}}{4h^2}\\
- & (1-\epsilon) \frac{(3a_{i-1,j-1}+a_{i-1,j+1})u_{i-1,j-1}+(a_{i-1,j-1}+3a_{i-1,j+1})u_{i-1,j+1}}{4h^2},\\
 \mathcal A_d^{-1}[\mathcal A^s (\bar{\mathbf u})]_{i-1,j} & =-(1-\epsilon)\frac{(3a_{i-2,j}+a_{ij})u_{i-2,j}+(a_{i-2,j}+3a_{ij})u_{i,j}}{4(a_{i-2,j}+a_{ij}+a_{i-1,j+1}+a_{i-1,j-1}+h^2c_{i-1,j})}\\
&-(1-\epsilon)\frac{(3a_{i-1,j-1}+a_{i-1,j+1})u_{i-1,j-1}+(a_{i-1,j-1}+3a_{i-1,j+1})u_{i-1,j+1}}{4(a_{i-2,j}+a_{ij}+a_{i-1,j+1}+a_{i-1,j-1}+h^2c_{i-1,j})}.
 \end{align*}
 
 Since the coefficient of
$u_{i-2,j}$ in $\mathcal A_a^+(\bar{\mathbf u})_{ij}$ is $(3a_{i-2,j}-4 a_{i-1,j}+3a_{ij})^+/(8h^2)$, we only need to discuss the case $3a_{i-2,j}-4 a_{i-1,j}+3a_{ij}>0$, for which the coefficient of
$u_{i-2,j}$ in $\mathcal A^z\left[\mathcal A_d^{-1}\left(\mathcal A^s(\bar{\mathbf u})\right)\right]_{ij}$ becomes 
\[ \frac{a_{i-2,j}+4 a_{i-1,j}+9a_{ij}}{8 h^2}\frac{(1-\epsilon)(3a_{i-2,j}+a_{ij})}{4(a_{i-2,j}+a_{ij}+a_{i-1,j+1}+a_{i-1,j-1}+h^2c_{i-1,j})}. \]
To ensure the coefficient of
$u_{i-2,j}$ in $\mathcal A^z\left[\mathcal A_d^{-1}\left(\mathcal A^s(\bar{\mathbf u})\right)\right]_{ij}$ is no less than the coefficient of
$u_{i-2,j}$ in $\mathcal A_a^{+}(\bar{\mathbf u})_{ij}$, we need  
\[\frac{(1-\epsilon)(a_{i-2,j}+4 a_{i-1,j}+9a_{ij})(3a_{i-2,j}+a_{ij})}{32 h^2(a_{i-2,j}+a_{ij}+a_{i-1,j+1}+a_{i-1,j-1}+h^2c_{i-1,j})}\geq \frac{3 a_{i-2,j}-4a_{i-1,j}+3a_{ij}}{8 h^2}.\]
Similar to the one-dimensional case, it suffices to require
\[\frac{(a_{i-2,j}+4 a_{i-1,j}+9a_{ij})(3a_{i-2,j}+a_{ij})}{4 (a_{i-2,j}+a_{ij}+a_{i-1,j+1}+a_{i-1,j-1}+h^2c_{i-1,j})}>3 a_{i-2,j}-4a_{i-1,j}+3a_{ij}.\]
Equivalently, we need the following inequality holds for any cell center $x_{ij}$: 
\begin{subequations}
 \label{h-condition-2d-cellcenter}
 \begin{align}
\label{h-condition-2d-cellcenter-1}
 \frac{(a_{i-1,j}+4 a_{i,j}+9a_{i+1,j})(3a_{i-1,j}+a_{i+1,j})}{4 (a_{i-1,j}+a_{i+1,j}+a_{i,j+1}+a_{i,j-1}+h^2c_{i,j})}>3 a_{i-1,j}-4a_{i,j}+3a_{i+1,j}.
\end{align}
Notice that \eqref{h-condition-2d-cellcenter-1} was derived for comparing $\mathcal A^z\left[\mathcal A_d^{-1}\left(\mathcal A^s(\bar{\mathbf u})\right)\right]_{ij}$ and $\mathcal A_a^{+}(\bar{\mathbf u})_{ij}$ for $x_{ij}$ being an edge center of an edge parallel to $y$-axis. 
If $x_{ij}$ is an edge center of an edge parallel to $x$-axis, then we can derive a similar constraint:
 \begin{align}
\label{h-condition-2d-cellcenter-2}
 \frac{(a_{i,j-1}+4 a_{i,j}+9a_{i,j+1})(3a_{i,j-1}+a_{i,j+1})}{4 (a_{i,j-1}+a_{i,j+1}+a_{i+1,j}+a_{i-1,j}+h^2c_{i,j})}>3 a_{i,j-1}-4a_{i,j}+3a_{i,j+1}.
\end{align}
\end{subequations}

If $x_{ij}$ is a knot, then $x_{i\pm1,j}$ are edge centers for an edge parallel to $x$-axis. Since everything here has a symmetric structure, we only need to compare the coefficients of
$u_{i-2,j}$ in $\mathcal A^z\left[\mathcal A_d^{-1}\left(\mathcal A^s(\bar{\mathbf u})\right)\right]_{ij}$ and
$\mathcal A_a^{+}(\bar{\mathbf u})_{ij}$, and the comparison for the coefficients of
$u_{i+2,j}$, $u_{i,j-2}$ and $u_{i,j+2}$ will be similar. 
\begin{align*}
 &\mathcal A^s (\bar{\mathbf u})_{i-1,j}=(1-\epsilon)\frac{-(3a_{i-2,j}+a_{i,j})u_{i-2,j}-(a_{i-2,j}+3a_{i,j})u_{i,j}}{4h^2}\\
&+\frac{-(3a_{i-1,j-2}-4a_{i-1,j-1}+3a_{i-1,j})^+u_{i-1,j-1}-(3a_{i-1,j+2}-4a_{i-1,j+1}+3 a_{i-1,j})^+u_{i-1,j+1}}{8h^2}
\\
 &\mathcal A_d^{-1}[\mathcal A^s (\bar{\mathbf u})]_{i-1,j}=(1-\epsilon)\frac{-(3a_{i-2,j}+a_{i,j})u_{i-2,j}-(a_{i-2,j}+3a_{i,j})u_{i,j}}{\frac{1}{2}(a_{i-1,j-2}+4a_{i-1,j-1}+18 a_{i-1,j}+4a_{i-1,j+1}+a_{i-1,j+2})+4(a_{i-2,j}+a_{i,j})+4h^2c_{i-1,j}}\\
&+\frac{-(3a_{i-1,j-2}-4a_{i-1,j-1}+3a_{i-1,j})^+u_{i-1,j-1}-(3a_{i-1,j+2}-4a_{i-1,j+1}+3 a_{i-1,j})^+u_{i-1,j+1}}{(a_{i-1,j-2}+4a_{i-1,j-1}+18 a_{i-1,j}+4a_{i-1,j+1}+a_{i-1,j+2})+8(a_{i-2,j}+a_{i,j})+8h^2c_{i-1,j}}.
 \end{align*}
 For the same reason as above we still only consider the case where $3a_{i-2,j}-4 a_{i-1,j}+3a_{ij}>0$. So the coefficient of
$u_{i-2,j}$ in $\mathcal A^z\left[\mathcal A_d^{-1}\left(\mathcal A^s(\bar{\mathbf u})\right)\right]_{ij}$ is 
\[ \frac{1}{4 h^2}\frac{(1-\epsilon)(a_{i-2,j}+4 a_{i-1,j}+9a_{ij})(3a_{i-2,j}+a_{i,j})}{(a_{i-1,j-2}+4a_{i-1,j-1}+18 a_{i-1,j}+4a_{i-1,j+1}+a_{i-1,j+2})+8(a_{i-2,j}+a_{i,j})+8c_{i-1,j}h^2}. \]
To ensure the coefficient of
$u_{i-2,j}$ in $\mathcal A^z\left[\mathcal A_d^{-1}\left(\mathcal A^s(\bar{\mathbf u})\right)\right]_{ij}$ is no less than the coefficient of
$u_{i-2,j}$ in $\mathcal A_a^{+}(\bar{\mathbf u})_{ij}$,  we only need  
\[\frac{2(a_{i-2,j}+4 a_{i-1,j}+9a_{ij})(3a_{i-2,j}+a_{i,j})}{(a_{i-1,j-2}+4a_{i-1,j-1}+18 a_{i-1,j}+4a_{i-1,j+1}+a_{i-1,j+2})+8(a_{i-2,j}+a_{i,j})+8c_{i-1,j}h^2}\]
\[> 3 a_{i-2,j}-4a_{i-1,j}+3a_{ij}.\]

Equivalently, we need the following inequality holds for any edge center  $x_{ij}$  for an edge parallel to $x$-axis: 
\begin{subequations}
 \label{h-condition-2d-edgecenter}
\begin{align}\label{h-condition-2d-xedgecenter}
&&\frac{2(a_{i-1,j}+4 a_{i,j}+9a_{i+1,j})(3a_{i-1,j}+a_{i+1,j})}{(a_{i,j-2}+4a_{i,j-1}+18 a_{i,j}+4a_{i,j+1}+a_{i,j+2})+8(a_{i-1,j}+a_{i+1,j})+8c_{i,j}h^2}\nonumber\\
&&> 3 a_{i-1,j}-4a_{i,j}+3a_{i+1,j}.
\end{align}
We also need the following inequality holds for any edge center  $x_{ij}$  for an edge parallel to $y$-axis: 
\begin{align}\label{h-condition-2d-yedgecenter}
&&\frac{2(a_{i,j-1}+4 a_{i,j}+9a_{i,j+1})(3a_{i,j-1}+a_{i,j-1})}{(a_{i-2,j}+4a_{i-1,j}+18 a_{i,j}+4a_{i+1,j}+a_{i+2,j})+8(a_{i,j-1}+a_{i,j+1})+8c_{i,j}h^2}\nonumber\\
&&> 3 a_{i,j-1}-4a_{i,j}+3a_{i,j+1}.
\end{align}
\end{subequations}
We have similar result to the one-dimensional case as following:
\begin{theorem}
 For the scheme \eqref{p2fd-vcoef-2dscheme} solving  $- \nabla\cdot (a \nabla u)+cu=f$ with $a(x)>0$ and $c(x)\geq 0$, its matrix representation $A=\bar{L}_h$ satisfies $A^{-1}\geq 0$ if \eqref{h-condition-2d-cellcenter} holds for any cell center $x_{ij}$, \eqref{h-condition-2d-xedgecenter} holds for $x_{ij}$ being  any edge center of an edge parallel to $x$-axis  and   \eqref{h-condition-2d-yedgecenter} holds for  $x_{ij}$ being any edge center of an edge parallel to $y$-axis.
\end{theorem}
The constraints \eqref{h-condition-2d-cellcenter},  \eqref{h-condition-2d-xedgecenter} and  \eqref{h-condition-2d-yedgecenter} can be satisfied for small $h$. 
\begin{theorem}\label{2d-thm-mesh-1}
For the scheme \eqref{p2fd-vcoef-2dscheme} solving  $-\nabla (a(x)\nabla u)+c u=f$ with $a(x)>0$ and $c(x)\geq 0$, its matrix representation $A=\bar{L}_h$ satisfies $A^{-1}\geq 0$ if  
the following mesh constraint is achieved for all edge centers $x_{ij}$:
\[\min_{J_{ij}} a(x)^2 >\frac{49}{61}\max_{J_{ij}}a(x)^2 + \frac{8}{61}\left(3\max_{J_{ij}}a(x)-2\min_{J_{ij}} a(x)\right)h^2c_{ij},\]
where $J_{ij}$ is the union of two finite element cells: if $x_{ij}$ is an edge center of an edge parallel to $x$-axis, then  $J_{ij}= [x_{i-1}, x_{i+1}]\times[y_{j-2}, y_{j+2}]$; if $x_{ij}$ is an edge center of an edge parallel to $y$-axis, then  $J_{ij}= [x_{i-2}, x_{i+2}]\times[y_{j-1}, y_{j+1}]$. 
\end{theorem}

\begin{theorem}
\label{2d-thm-mesh-2}
  For the scheme \eqref{p2fd-vcoef-2dscheme} solving  $- \nabla\cdot (a \nabla u)+cu=f$ with $a(x)>0$ and $c(x)\geq 0$ on a uniform mesh, its matrix representation $A=\bar{L}_h$ satisfies $A^{-1}\geq 0$ if  any of  the following mesh constraints 
  is satisfied for any edge center $x_{ij}$:
  \begin{itemize}
   \item There exists some $\lambda\in(\frac{49}{61},1)$ such that 
   \[h^2 c_{ij}< \frac{61(1-\lambda)\min\limits_{J_{ij}} a^2(x)}{8\left(3\max\limits_{J_{ij}} a(x)-2\min\limits_{J_{ij}} a(x)\right)},\qquad h\frac{\max\limits_{x\in J{ij}} \left|\nabla a(x)\right |}{\min\limits_{x\in J_{ij}}a(x)}<\frac{\sqrt{122\lambda}-7\sqrt2}{28}.\]
\item   
$ \frac{49\sqrt2}{3} h \max\limits_{J_{ij}} |\nabla a(x)|+2h^2 c_{ij} \left(1-\frac23\frac{\min\limits_{J_{ij}} a(x)}{\max\limits_{J_{ij}} a(x)}\right)< \frac{\min\limits_{J_{ij}} a^2(x)}{\max\limits_{J_{ij}} a(x)}.$
  \item  If $c(x)\equiv0$, then we only need
    $ h\frac{\max\limits_{x\in J_{ij}} \left|\nabla a(x)\right |}{\min\limits_{x\in J_{ij}}a(x)}<\frac{\sqrt{122}-7\sqrt2}{28}.$
 \item If $a(x)\equiv a>0$, then we only need 
    $h^2 c_{ij}< \frac32 a.$
  \end{itemize}
Here the definition of $J_{ij}$ is the same as in Theorem \ref{2d-thm-mesh-1}.   
\end{theorem}
The proof of Theorem \ref{2d-thm-mesh-1} is included in the Appendix \ref{appendix-b}. The proof of Theorem \ref{2d-thm-mesh-2} is very similar to the proof of Theorem \ref{1d-thm-mesh-1} thus omitted. Since the two-dimensional case is more complicated, it does not seem possible to derive a similar mesh constraint involving second order derivatives of $a(x,y)$ as in Theorem \ref{1d-thm-mesh-2}. 
For instance, by Theorem \ref{1d-thm-mesh-2}, if $a(x)>0$ is concave and $c(x)\equiv 0$, then the one-dimensional scheme \eqref{p2fd-vcoef-1dscheme} satisfies $\bar L_h^{-1}\geq 0$ without any mesh constraint. 
For the two-dimensional scheme \eqref{p2fd-vcoef-2dscheme}, even if assuming $a(x,y)>0$ is concave and $c(x,y)\equiv 0$, constraints \eqref{h-condition-2d-cellcenter},  \eqref{h-condition-2d-xedgecenter} and  \eqref{h-condition-2d-yedgecenter}  are not all satisfied for any $h$. 

\section{Numerical test}
\label{sec-test}
In this section we show some numerical tests of scheme \eqref{p2fd-vcoef-2dscheme} on an uniform rectangular mesh and verify the inverse non-negativity of $\mathcal L_h$. See \cite{li2019fourth} for numerical tests on the fourth order accuracy of this scheme. 
In order to minimize round-off errors, we redefine \eqref{p2fd-vcoef-2dscheme-1} to its equivalent expression $
 \mathcal L_h(\bar{\mathbf u})_{i,j}=\frac{1}{h^2}u_{i,j}=\frac{1}{h^2}g_{i,j}$
 so that all nonzero entries in $\bar{L}_h$ have similar magnitudes. 
By Theorem \ref{theorem-fullandsmallmatrix}, we have $L_h^{-1}
\geq 0$  whenever $\bar{L}_h^{-1}\geq 0$. Even though $L_h^{-1}
\geq 0$ is not sufficient to ensure the discrete maximum principle, in practice only $L_h^{-1}$ is used directly thus its positivity is also important.

We first consider the following equation with purely Dirichlet conditions:  
\begin{equation}\label{numtestpoisson}
-\nabla\cdot(a\nabla u) + c u=f\quad \textrm{on } [0,1]\times [0,2]
\end{equation}
where $c(x)\equiv 10$ and $a(x,y)=1+d\cos(\pi x)\cos(\pi y)$ with $d= 0.5, 0.9$, and $0.99$. 
The smallest entries in $L_h^{-1}$ and $\bar{L}_h^{-1}$ are listed in Table \ref{poisson}, in which $-10^{-18}$ should be regarded as the numerical zero.    
As we can see,   $L_h^{-1}
\geq 0$ and $\bar{L}_h^{-1}\geq 0$ are achieved when $h$ is small enough. 
\begin{table}[h]
\label{poisson}
\centering
\caption{Minimum of entries in $\bar L_h^{-1}$ and $L_h^{-1}$  for Poisson equation \eqref{numtestpoisson} with smooth coefficients.}
{\renewcommand{\arraystretch}{1.2}\resizebox{\textwidth}{!}{\begin{tabular}{|c|c|c|c|c|c|c|}
\hline
\multirow{2}{*}{Finite Element Mesh} & \multicolumn{2}{c|}{$d=0.5$} & \multicolumn{2}{c|}{$d=0.9$} & \multicolumn{2}{c|}{$d=0.99$} \\ \cline{2-7}
     & $\bar L_h^{-1}$ & $L_h^{-1}$ & $\bar L_h^{-1}$ & $L_h^{-1}$& $\bar L_h^{-1}$ & $L_h^{-1}$ \\ \hline
$2\times 4$  & -7.32E-18  & 7.48E-06 & -3.90E-04 & 6.37E-06 & -7.41E-04 &  6.14E-06 \\ \hline
 $4\times 8$ & -1.31E-18 & 1.23E-07 & -4.02E-19& 9.95E-08 & -1.65E-04 &  9.44E-08\\ \hline
$8\times 16$ &  -3.96E-19 &  1.91E-09& -4.91E-19 & 1.52E-09& -1.77E-05 & 1.44E-09\\ \hline
$16\times 32$   & -1.92E-19 & 2.98E-11 & -7.60E-19 & 2.35E-11 & -1.06E-18 & 2.22E-11\\
  \hline
\end{tabular}}}
\end{table}

Next we consider \eqref{p2fd-vcoef-2dscheme}  solving \eqref{numtestpoisson} with $c(x,y)\equiv 0$
and $a_{ij}$ being random uniformly distributed random numbers in the interval $(d,d+1)$. Notice that the larger $d$ is, the smaller $\frac{\max\limits_{ij}\{a_{ij}\}}{\min\limits_{ij}\{a_{ij}\}}$ is. When $d=10$, we  have $\frac{\max\limits_{ij}\{a_{ij}\}}{\min\limits_{ij}\{a_{ij}\}} < \sqrt{\frac{61}{49}},$ thus $L_h^{-1}
\geq 0$ and $\bar{L}_h^{-1}\geq 0$ are guaranteed by Theorem \ref{2d-thm-mesh-1}. In Table \ref{rancoef} we can see that the upper bound on $\frac{\max\limits_{ij}\{a_{ij}\}}{\min\limits_{ij}\{a_{ij}\}}$ is indeed a necessary condition to have $\bar{L}_h^{-1}\geq 0$, even though constraints in Theorem \ref{2d-thm-mesh-1} may not be sharp since we still have the positivity when $d=1$. We have tested $d= 0.3$ many times and never observed negative entries in $\bar L_h^{-1}$ and $L_h^{-1}$.
\begin{table}[h]
\label{rancoef}
\centering
\caption{Minimum of all entries of $\bar L_h^{-1}$ and $L_h^{-1}$ for $a(x,y)$ being random coefficients }\resizebox{\textwidth}{!}{
{\renewcommand{\arraystretch}{1.2}\begin{tabular}{|c|c|c|c|c|c|c|}
\hline
\multirow{2}{*}{Finite Element Mesh} & \multicolumn{2}{c|}{$d=0.1$} & \multicolumn{2}{c|}{$d=1$} & \multicolumn{2}{c|}{$d=10$} \\\cline{2-7}
     & $\bar L_h^{-1}$ & $L_h^{-1}$ & $\bar L_h^{-1}$ & $L_h^{-1}$& $\bar L_h^{-1}$ & $L_h^{-1}$ \\ \hline
$2\times 4$  & -1.00E-03  & 6.60E-05 & -8.15E-18 & 4.73E-05
 & -1.98E-16 &  6.74E-06 \\ \hline
 $4\times 8$ & -2.14E-04 & 3.22E-06 & -3.46E-18& 9.95E-07 & -5.10E-17 &  1.35E-07\\ \hline
$8\times 16$ &  -6.73E-05 & 2.88E-08 & -5.24E-19 & 1.65E-08&  -1.81E-17& 2.21E-09\\ \hline
$16\times 32$   &-2.34E-05 & 3.61E-10 & -9.01E-19& 2.02E-10 & -8.37E-18 & 3.56E-11\\
  \hline
\end{tabular}}}
\end{table}

Last we consider solving the heat equation 
$
u_t=\Delta u$ on $[0,1]\times [0,2]
$
with backward Euler time discretization $
- \Delta u^{n+1} + \frac{1}{\Delta t}u^{n+1}  =\frac{u^n}{\Delta t},$
corresponding to \eqref{numtestpoisson} with $a(x,y)\equiv 1$ and $c=\frac{1}{\Delta t}$. 
By Theorem \ref{2d-thm-mesh-2},  
$
\frac{\Delta t}{h^2} > \frac{2}{3},
$
is a sufficient condition to ensure $\bar L_h^{-1}\geq0$ and $L_h^{-1}\geq0$. In Table \ref{heateqn}, we can see that it is necessary to have a lower bound constraint on $\frac{\Delta t}{h^2}$ but
$\frac{\Delta t}{h^2} > \frac{2}{3}$ is not sharp at all.  In Figure \ref{minimumfig},
we can see the minimum of  entries in $\bar L_h^{-1}$ and $L_h^{-1}$ decreases  for smaller $\frac{\Delta t}{h^2}$. The lower bound to ensure the inverse non-negativity of $\bar L_h^{-1}$ and $L_h^{-1}$ seems to be near $\frac{\Delta t}{h^2}=\frac{1}{3.6}$.

\begin{table}[h]
\label{heateqn}
\centering
\caption{Minimum of all entries of $\bar L_h^{-1}$ and $L_h^{-1}$  for solving heat equation with backward Euler.}
{\renewcommand{\arraystretch}{1.2}\begin{tabular}{|c|c|c|c|c|c|c|}
\hline
\multirow{2}{*}{Finite Element Mesh} & \multicolumn{2}{c|}{$\Delta t = \frac{3 h^2}{2}$} & \multicolumn{2}{c|}{$\Delta t= \frac{h^2}{2}$} & \multicolumn{2}{c|}{$\Delta t= \frac{h^2}{4}$} \\ \cline{2-7}
     & $\bar L_h^{-1}$ & $L_h^{-1}$ & $\bar L_h^{-1}$ & $L_h^{-1}$& $\bar L_h^{-1}$ & $L_h^{-1}$ \\ \hline
$2\times 4$  & 0  & 7.95E-06 & 0 & 3.21E-07 & -9.14E-05 &  -5.34E-07 \\ \hline
 $4\times 8$ & 0 & 1.01E-09 & 0 & 1.93E-13 & -2.28E-05 &  -1.00E-07\\ \hline
$8\times 16$ & 0 &  7.74E-17& 0 & 2.58E-25 & -5.71E-06 & -2.51E-08\\ \hline
$16\times 32$& 0 & 2.63E-30 & 0 & 2.73E-48 & -1.43E-06 & -6.27E-09\\
  \hline
\end{tabular}}
\end{table}
 
\begin{figure}
\label{minimumfig}
  \subfigure[Minimum of entries in $\bar L_h^{-1}$]{\includegraphics[scale=0.33]{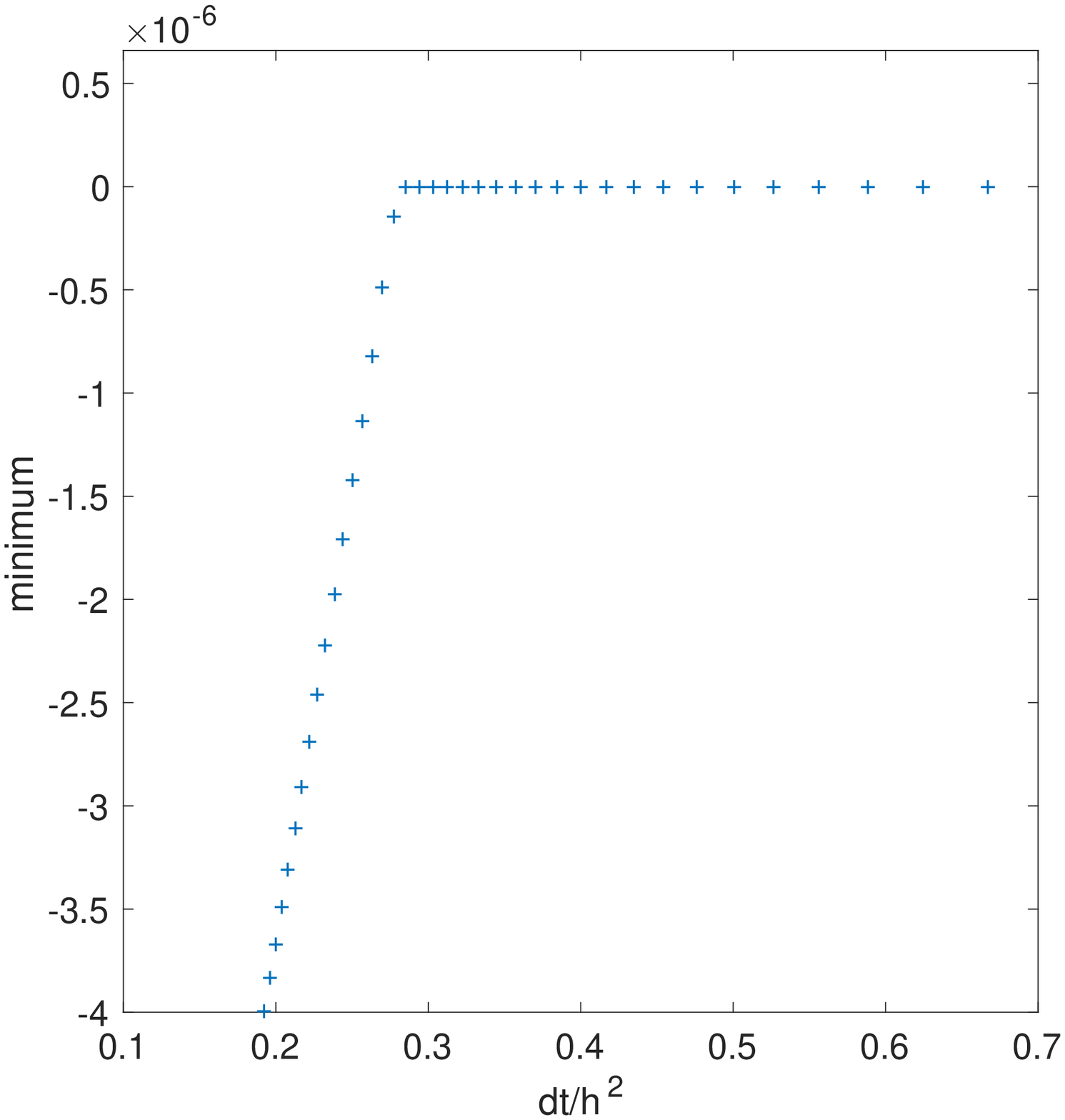}}
   \subfigure[Minimum of  entries in $L_h^{-1}$]{\includegraphics[scale=0.33]{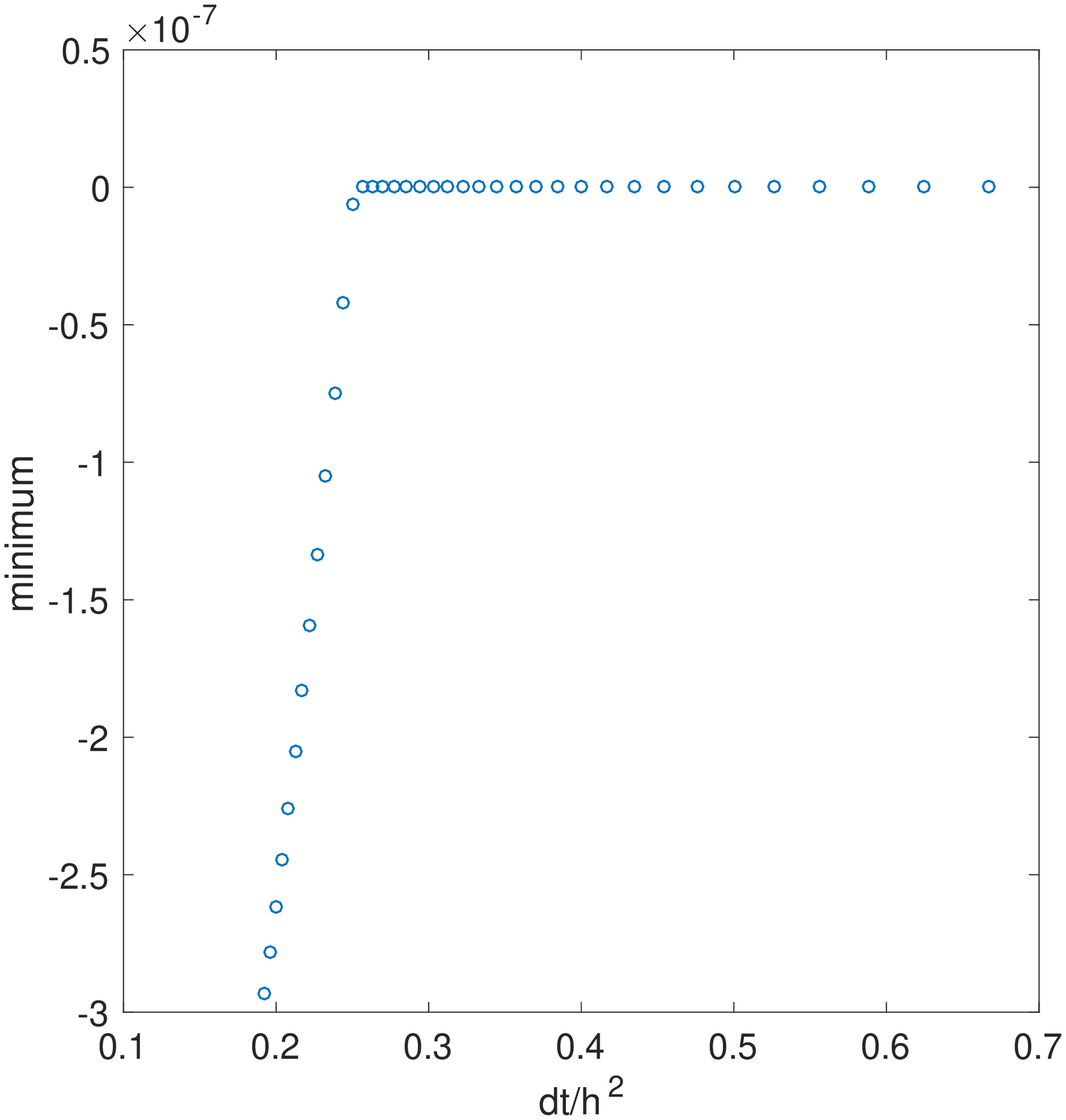}}
   \caption{Minimum of all entries of $\bar L_h^{-1}$ and $L_h^{-1}$ on $16\times 32$ mesh with different time steps.}
\end{figure}

\section{Concluding remarks}
\label{sec-remark}

In this paper we have proven that the simplest fourth order accurate finite difference implementation of $C^0$-$Q^2$ finite element method is monotone thus satisfies a discrete maximum principle for solving a variable coefficient problem $-\nabla \cdot (a(x,y)\nabla u)+c(x,y) u=f$ under some suitable mesh constraints. 
The main results in this paper can be used to construct high order spatial discretization preserving positivity or maximum principle for solving time-dependent diffusion problems implicitly by backward Euler time discretization. 

\appendix
\section{M-Matrix factorization for discrete Laplacian}
\label{appendix-a}
The matrix form of \eqref{p2fd-laplacian-1dscheme} can be written as 
$\frac{1}{h^2} \bar L_h\bar{\mathbf u}=\bar{\mathbf f}$. As  an example, if there are seven interior grid points in the mesh for $(0,1)$, then  the matrix $\bar L_h$ is given by  
\[\bar L_h=\begin{psmallmatrix}
1 &&&&&&&\\
 -1&  2& -1 & & & & & &\\
  \frac14& -2& \frac72 &-2 & \frac14 & & & &\\
  &   &  -1 & 2& -1 & & & &\\
 &  & \frac14 &-2& \frac72 &-2 & \frac14 & &\\
 & & &   &  -1 & 2& -1 & &\\
 & && & \frac14 &-2& \frac72 &-2& \frac14\\
 &  &  & & & &-1 & 2 & -1\\
  &  &  & & & & &  & 1
  \end{psmallmatrix}
\]
The matrix $\bar L_h$ can be written as a product of two nonsingular M-matrices $\bar L_h=M_1 M_2$ where
\[M_1=\begin{psmallmatrix}
 1 & &  & & & & & & \\
 &   1&  & & & & & &\\
 &  -\frac14& 1 &-\frac14 &  & & & &\\
 &    &   & 1&  & & & &\\
 && &-\frac14& 1 &-\frac14 &  & & \\
 &  & & &   & 1&  & & \\
 &&&  &&-\frac14& 1 &-\frac14 &    \\
 & & && & &   & 1  & \\
 & & && & &   &   & 1
   \end{psmallmatrix}, M_2=\begin{psmallmatrix}
 1 &  &  & & & & & & &\\
 -1 &  2& -1 & & & & & & &\\
  & -\frac32& 3 &-\frac32 &  & & & &\\
  &   &  -1 & 2& -1 & & & &\\
 && &-\frac32& 3 &-\frac32 &  & & \\
 &  & & & -1  & 2& -1 & &\\
 &&&  &&-\frac32& 3 &-\frac32  &   \\
 && && & &  -1 & 2 & -1  \\
  && && & &   &  & 1 
   \end{psmallmatrix}.\]
   Such a factorization is  not unique and it does not seem to have further physical or geometrical meanings.
   


For the scheme \eqref{q2fd-2D-laplacian}, we can find  two linear operators $\mathcal A_1$ and $\mathcal A_2$ are  with their matrix representations $A_1$ and $A_2$ being nonsingular M-matrices, such that
$\mathcal L_h(\bar{\mathbf u})=\mathcal A_2 (\mathcal A_1 (\bar{\mathbf u}))$. 

Definition of $\mathcal A_1$ is given as
\begin{itemize}
\item At boundary points:
\[v_{i,j}=\mathcal A_1(\bar{\mathbf u})_{i,j}=u_{i,j}:=g_{ij}.\]
 \item At interior knots:
\begin{equation*}
 v_{i,j}=\mathcal A_1(\bar{\mathbf u})_{i,j}=u_{i,j}.
\end{equation*}
\item
At interior cell center:
\begin{equation*}
v_{i,j}=\mathcal A_1(\bar{\mathbf u})_{i,j}= 2u_{i,j}-\frac14 u_{i-1,j}-\frac14 u_{i+1,j}-\frac14u_{i,j-1}-\frac14u_{i,j+1}.
\end{equation*}
\item 
At interior edge center (an edge parallel to x-axis):
\begin{equation*}
v_{i,j}=\mathcal A_1(\bar{\mathbf u})_{i,j}= -\frac16 u_{i-1,j}+\frac43 u_{i,j}-\frac16 u_{i+1,j}.
\end{equation*}
\item 
At interior edge center (an edge parallel to y-axis):
\begin{equation*}
v_{i,j}=\mathcal A_1(\bar{\mathbf u})_{i,j}= -\frac16 u_{i,j-1}+\frac43 u_{i,j}-\frac16 u_{i,j+1}.
\end{equation*}
\end{itemize}

Definition of $\mathcal A_2$ is given as:
\begin{itemize}
\item
At boundary points:
\[ \mathcal A_2(\bar{\mathbf v})_{i,j}=v_{i,j}.\]
 \item At an interior knot:
\begin{equation*}
 \mathcal A_2(\bar{\mathbf v})_{i,j}=  -\frac32 v_{i-1,j}+3 v_{i,j}- \frac32 v_{i+1,j}- \frac32 v_{i,j-1}+3  v_{i,j}- \frac32  v_{i,j+1} 
 \end{equation*}
\item At an interior cell center:
\begin{equation*}
\mathcal A_2(\bar{\mathbf v})_{i,j}= 2v_{i,j} -\frac38 v_{i-1,j}-\frac38 v_{i+1,j}-\frac38v_{i,j-1}-\frac38v_{i,j+1}-\frac18 v_{i-1,j+1}-\frac18 v_{i+1,j+1}-\frac18v_{i-1,j-1}-\frac18v_{i+1,j+1}.
\end{equation*}
\item At an interior edge center (an edge parallel to x-axis):
\begin{align*}
\mathcal A_2(\bar{\mathbf v})_{i,j} = -\frac{7}{16} v_{i-1,j}+\frac{15}{4}v_{i,j} -\frac{7}{16}  v_{i+1,j} -v_{i,j+1}-v_{i,j-1}-\frac{3}{16}v_{i-1,j-1}-\frac{3}{16}v_{i+1,j-1}\\
-\frac{3}{16}v_{i-1,j+1}-\frac{3}{16}v_{i+1,j+1}-\frac{1}{32}v_{i-1,j+2}-\frac{1}{32}v_{i+1,j+2}-\frac{1}{32}v_{i-1,j-2}-\frac{1}{32}v_{i+1,j-2}.\\
\end{align*}
\item At an interior edge center (an edge parallel to y-axis):
\begin{align*}
\mathcal A_2(\bar{\mathbf v})_{i,j} = -\frac{7}{16} v_{i,j-1}+\frac{15}{4}v_{i,j} -\frac{7}{16}  v_{i,j+1} -v_{i+1,j}-v_{i-1,j}-\frac{3}{16}v_{i-1,j-1}-\frac{3}{16}v_{i-1,j+1}\\
-\frac{3}{16}v_{i+1,j-1}-\frac{3}{16}v_{i+1,j+1}-\frac{1}{32}v_{i+2,j-1}-\frac{1}{32}v_{i+2,j+1}-\frac{1}{32}v_{i-2,j-1}-\frac{1}{32}v_{i-2,j+1}.\\
\end{align*}
\end{itemize}

It is straightforward to verify that $\mathcal L_h(\bar{\mathbf u})=\mathcal A_2(\bar{\mathbf v})$ where $\bar{\mathbf v}=\mathcal A_1 (\bar{\mathbf u})$. Obviously, matrices of $\mathcal A_1$ and  $\mathcal A_2$ have positive diagonal entries and nonpositive off-diagonal entries. Moreover, $\mathcal A_1(\mathbf 1)\geq 0$ and $\mathcal A_2(\mathbf 1)\geq 0$ thus $A_1$ and $A_2$
satisfy the row sum conditions in Theorem \ref{rowsumcondition-thm}. So $A_1$ and $A_2$ are both nonsingular $M$-matrices and the matrix representation of $\mathcal L_h$ is $A_2A_1$. However, this kind of M-matrix factorization cannot be extended to the variable coefficient case.

\section{}
\label{appendix-b}

\begin{proof}[Proof of Theorem \ref{1d-thm-mesh-1}]
 If $c(x)\equiv 0$, then  \eqref{h-condition-2}  reduces to
\[(28 a_{i-1}+20 a_{i+1}) a_i +4a_{i+1}a_{i-1} >9a^2_{i-1}+ 3a^2_{i+1}.
\]
A convenient sufficient condition is to require
\[     52 \min \{a_{i-1}^2,a_i^2, a^2_{i+1}\} >12 \max \{a^2_{i-1},a^2_i, a^2_{i+1}\},\]
which is equivalent to 
\[ \frac{\max \{a_{i-1},a_i, a_{i+1}\}}{\min \{a_{i-1},a_i, a_{i+1}\}}<\sqrt{\frac{13}{3}}.\]
Let $a(x^1)=\max \{a_{i-1},a_i, a_{i+1}\}$ and $a(x^2)=\min \{a_{i-1},a_i, a_{i+1}\}$. Then the inequality above is equivalent to
\[\frac{a(x^1)-a(x^2)}{a(x^2)}< \frac{\sqrt{39}-3}{3}.\]
By the Mean Value Theorem, there is some $\xi\in (x_{i-1}, x_{i+1})$ such that 
$a(x^1)-a(x^2)=a'(\xi)(x^2-x^1)$. Since $|x^2-x^1|\leq 2h$, we have 
$$|a(x^1)-a(x^2)|\leq \max\limits_{x\in (x_{i-1}, x_{i+1})} \left|a'(x)\right |2h.$$
Thus a sufficient condition is to require
\[ h\frac{\max\limits_{x\in (x_{i-1}, x_{i+1})} \left|a'(x)\right |}{\min\limits_{x\in (x_{i-1}, x_{i+1})}a(x)}<\frac{\sqrt{39}-3}{6}. \]

For $c(x)\geq 0$, 
 \eqref{h-condition-2}  reduces to
\[(28 a_{i-1}+20 a_{i+1}) a_i +4a_{i+1}a_{i-1} >9a^2_{i-1}+ 3a^2_{i+1}+4h^2 c_i(3a_{i-1}-4a_i +3a_{i+1}),
\]
for which a sufficient condition is
\begin{equation}
13\min\limits_{I_i} a^2(x) >3 \max\limits_{I_i} a^2(x)+h^2 c_i (6\max\limits_{I_i} a(x)-4\min\limits_{I_i} a(x)). 
\label{appendix-cond1}
 \end{equation}
One sufficient condition for \eqref{appendix-cond1} is to have
\[\exists \lambda\in(0,1), h^2 c_i (6\max\limits_{I_i} a(x)-4\min\limits_{I_i} a(x))<13 (1-\lambda)\min\limits_{I_i} a^2(x), \quad 3 \max\limits_{I_i} a^2(x)<13\lambda \min\limits_{I_i} a^2(x). \]
By similar discussions above, a sufficient condition for  $ 3 \max\limits_{I_i} a^2(x)<13\lambda \min\limits_{I_i} a^2(x)$ is to have $\lambda>\frac{3}{13}$ and
\[ h\frac{\max\limits_{x\in I_i} \left|a'(x)\right |}{\min\limits_{x\in I_i}a(x)}<\frac{\sqrt{39\lambda}-3}{6}.\]
The inequality \eqref{appendix-cond1}  is also equivalent to
\[  10 \min\limits_{I_i} a^2(x) >3 (\max\limits_{I_i} a^2(x)-\min\limits_{I_i} a^2(x))+h^2 c_i (6\max\limits_{I_i} a(x)-4\min\limits_{I_i} a(x)).\]
Let $a^2(x^1)=\max\limits_{I_i} a^2(x)$ and $a^2(x^2)=\min\limits_{I_i} a^2(x)$, then by the Mean Value Theorem on the function $a^2(x)$, there is some $\xi\in (x_{i-1}, x_{i+1})$ such that
$$a^2(x^1)-a^2(x^1)=2a(\xi)a'(\xi)(x^1-x^2)\leq 4h\max\limits_{I_i} a(x)\max\limits_{I_i} |a'(x)|.$$
So it suffices to have
\[  10 \min\limits_{I_i} a^2(x) >12h\max\limits_{I_i} a(x)\max\limits_{I_i} |a'(x)|+h^2 c_i (6\max\limits_{I_i} a(x)-4\min\limits_{I_i} a(x)),\]
which can be simplified to 
  \[ 2 h \max\limits_{I_i} |a'(x)|+h^2 c_i(1-\frac23\frac{\min\limits_{I_i} a(x)}{\max\limits_{I_i} a(x)}) < \frac{5}{3} \frac{\min\limits_{I_i} a^2(x)}{\max\limits_{I_i} a(x)}.\]
If $a(x)\equiv a>0$, it is straightforward to verify that  \eqref{h-condition-2} is equivalent to $h c_i< 5a.$
\end{proof}

\begin{proof}[Proof of Theorem \ref{1d-thm-mesh-2}]
For a smooth coefficient $a(x)$, by Taylor's Theorem, 
\[a(x+h)=a(x)+ha'(x)+\frac12h^2 a''(\xi_1), \xi_1\in[x,x+h],\]
\[a(x-h)=a(x)-ha'(x)+\frac12h^2 a''(\xi_2), \xi_2\in[x-h,x].\]
With the Intermediate Value Theorem for $a''(x)$, we get
\[a(x)=\frac12[a(x+h)+a(x-h)-h^2 a''(\xi)],\quad \xi\in(\xi_2, \xi_1)\subset [x-h, x+h].\]
Thus we can rewrite $a_i$  as 
$a_i=\frac12(a_{i-1}+a_{i+1}-d_i h^2)$
where $$d_i:=\frac{a_{i-1}+a_{i+1}-2a_i}{h^2}=a''(\xi), \textrm{ for some }\xi\in(x_{i-1}, x_{i+1}).$$ 
If $c(x)\equiv 0$, then \eqref{h-condition-2} reduces to
$(28 a_{i-1}+20 a_{i+1}) a_i +4a_{i+1}a_{i-1} >9a^2_{i-1}+ 3a^2_{i+1}.
$
Introducing an arbitrary number $\lambda\in(0,2]$, it is equivalent to 
\begin{align*}
4a_{i+1}a_{i-1}+(4-2\lambda)a_i(7 a_{i-1}+5 a_{i+1}) +2\lambda a_i(7 a_{i-1}+5 a_{i+1}) > 9a^2_{i-1}+ 3a^2_{i+1,}& \\
(12\lambda + 4)a_{i+1}a_{i-1} + (4-2\lambda)a_i(7 a_{i-1}+5 a_{i+1}) + (7\lambda -9)a_{i-1}^2 + (5\lambda -3)a_{i+1}^2 &\\
>\lambda  h^2d_i(7a_{i-1}+5a_{i+1}), &\\
(\frac{4}{\lambda }-2)a_i + a_{i-1}\frac{(5\lambda -3)\theta^2+(12\lambda + 4)\theta + (7\lambda -9)}{\lambda (5\theta + 7)} >  h^2d_i,\quad \theta =\frac{a_{i+1}}{a_{i-1}},& \\
\left(\frac{4}{\lambda}-2\right)a_i + \left(\frac{\frac{41}{5}\theta - 9}{\lambda(5\theta + 7)} +1\right )a_{i-1}+ \left(1-\frac{3}{5\lambda }\right)a_{i+1} >  h^2d_i.&
\end{align*}
Notice that $\frac{\frac{41}{5}\theta - 9}{5\theta + 7} > -\frac97$. By taking $\frac97 \leq \lambda \leq 2$, it suffices to require
\begin{equation}\label{h-condition-3}
(1-\frac{9}{7\lambda})a_{i-1}+ (\frac{4}{\lambda}-2)a_i +  (1-\frac{3}{5\lambda })a_{i+1} >  h^2d_i,
\end{equation}
as a sufficient condition of the above inequalities.
If $a(x)$ is a concave function, then it satisfies 
$a(x_i)=a(\frac{x_{i-1}+x_{i-1}}{2})\geq\frac12a(x_{i-1})+\frac12a(x_{i+1}),$
which implies $a_{i-1}+a_{i+1}-2a_i\leq 0$,
thus
\eqref{h-condition-3} holds trivially. Otherwise,   \eqref{h-condition-3} holds for $\lambda = \frac{9}{7}$ if the following mesh constraint is satisfied:
\[ h^2\max_{x\in(x_{i-1}, x_{i+1})} a''(x) < \frac{74}{45}\min \{a_{i-1},a_i, a_{i+1}\}.\]

If $c(x)\geq 0$, for any $\lambda\in(0,2]$, \eqref{h-condition-2} is equivalent to
\begin{align}
(12\lambda + 4)a_{i+1}a_{i-1} + (4-2\lambda)a_i(7 a_{i-1}+5 a_{i+1}) + (7\lambda -9)a_{i-1}^2 + (5\lambda -3)a_{i+1}^2 & \nonumber\\
> \lambda h^2d_i(7a_{i-1}+5a_{i+1})  +4h^2 c_i(a_{i-1}+a_{i+1}+2 d_i h^2).  &\label{suf1}
\end{align}
If
assuming $ d_i h^2 \leq \frac{74}{45}\min\{a_{i-1},a_i,a_{i+1}\}$, then $ d_i h^2 \leq\lambda_1 a_{i-1}+\lambda_2 a_{i+1}$ for any two positive numbers $\lambda_1, \lambda_2$ satisfying $\lambda_1+\lambda_2=\frac{74}{45}$. 
In particular, for $\lambda_1=\frac{563}{540}$, we get
$  d_i h^2 \leq\frac{563}{540}a_{i-1}+\frac{65}{108}a_{i+1}$, which implies
\[
a_{i-1}+a_{i+1}+2 d_i h^2 \leq \frac{119}{270}(7a_{i-1}+5a_{i+1}).
\]
By replacing $a_{i-1}+a_{i+1}+2 d_i h^2$ by the inequality above in  \eqref{suf1}, we get a sufficient condition for \eqref{suf1} as following:
\begin{align}
(12\lambda + 4)a_{i+1}a_{i-1} + (4-2\lambda)a_i(7 a_{i-1}+5 a_{i+1}) + (7\lambda -9)a_{i-1}^2 + (5\lambda -3)a_{i+1}^2 & \nonumber\\
> \lambda h^2d_i(7a_{i-1}+5a_{i+1})  +4h^2 c_i\frac{119}{270}(7a_{i-1}+5a_{i+1}).  &\label{suf2}
\end{align}
Similar to the derivation of \eqref{h-condition-3}, we can derive a sufficient condition of \eqref{suf2} as
\[  h^2\left(1.5 c_i+\max_{x\in(x_{i-1}, x_{i+1})} a''(x)\right) < \frac{74}{45} \min\{a_{i-1},a_i,a_{i+1}\}.\]
If $d_i\leq 0$, then a sufficient condition for \eqref{suf1} is
\[
\frac{(12\lambda + 4)a_{i+1}a_{i-1} + (4-2\lambda)a_i(7 a_{i-1}+5 a_{i+1}) + (7\lambda -9)a_{i-1}^2 + (5\lambda -3)a_{i+1}^2}{a_{i-1}+a_{i+1}} \\
>4h^2 c_i,\]
from which we can derive a sufficient condition  as
\[4h^2 c_i < (7\lambda - 9 )a_{i-1}+(5-\frac52\lambda)a_i +(5\lambda -3)a_{i+1},\]
for which a sufficient condition by setting $\lambda=2$ is 
$
h^2 c_i < 3 \min\{a_{i-1},a_i,a_{i+1}\}.
$\end{proof}

\begin{proof}[Proof of Theorem \ref{2d-thm-mesh-1}]
Since \eqref{h-condition-2d-cellcenter-1} and \eqref{h-condition-2d-xedgecenter} are equivalent to
\begin{align*}
4(7a_{i-1,j}+5a_{i+1,j})a_{ij}+4a_{i-1,j}a_{i+1,j}+16a_{ij}(a_{i,j-1}+a_{i,j+1})\\
> 9 a^2_{i-1,j} +3 a^2_{i+1,j} + 12 (a_{i-1,j}+a_{i+1,j})(a_{i,j-1}+a_{i,j+1})+4(3a_{i-1,j}-4a_{ij}+3a_{i+1,j})h^2c_{ij}
\end{align*}
and 
\begin{align*}
8a_{i-1,j}a_{i+1,j}+2a_{ij}a_{i-1,j}+4a_{ij}(a_{i,j-2}+4a_{i,j-1}+18a_{i,j}+4a_{i,j+1}+a_{i,j+2})
> 18 a^2_{i-1,j} +6 a^2_{i+1,j} \\
+ 14a_{ij}a_{i+1,j} + 3(a_{i-1,j}+a_{i+1,j})(a_{i,j-2}+4a_{i,j-1}+4a_{i,j+1}+a_{i,j+2})+8(3a_{i-1,j}-4a_{ij}+3a_{i+1,j})h^2c_{ij}.
\end{align*}
A sufficient condition is to require 
 \begin{equation}
 \label{appendix-b-add1}
7\min_{I_{ij}} a(x)^2 >5\max_{I_{ij}}a(x)^2 + \frac{2}{3}(3\max_{I_{ij}}a(x)-2\min_{I_{ij}} a(x))h^2c_{ij}
\end{equation}
for  all cell centers $x_{ij}$ of cell $I_{ij}=[x_{i-1}, x_{i+1}]\times [y_{i-1}, y_{i+1}]$, 
and the following mesh constraints for all edge centers $x_{ij}$:
\begin{equation}
 \label{appendix-b-add2}
61\min_{J_{ij}} a(x)^2 >49\max_{J_{ij}}a(x)^2 + 8(3\max_{J_{ij}}a(x)-2\min_{J_{ij}} a(x))h^2c_{ij},
\end{equation}
where we $J_{ij}$ is the union of two cells: if $x_{ij}$ is an edge center of an edge parallel to $x$-axis, then  $J_{ij}= I_{i,j-1}\cup I_{i,j+1}$; if $x_{ij}$ is an edge center of an edge parallel to $y$-axis, then  $J_{ij}= I_{i-1,j}\cup I_{i+1,j}$.
Notice that \eqref{appendix-b-add2}
implies \eqref{appendix-b-add1}, thus it suffices to have \eqref{appendix-b-add2}
only. 
\end{proof}


\bibliographystyle{siamplain}

\bibliography{references.bib}

\end{document}